\documentclass[12pt]{amsart}
\usepackage{geometry}
\usepackage[colorlinks,citecolor = red, linkcolor=blue,hyperindex]{hyperref}
\usepackage{euscript,verbatim, mathrsfs}
\usepackage[psamsfonts]{amssymb}
\usepackage{bbm}
\usepackage{graphicx}
\usepackage{float}
\usepackage{float, tikz}
\usepackage{extpfeil}
\usepackage[all, cmtip]{xy}
\usepackage{upref, xcolor, dsfont}
\usepackage{amsfonts,amsmath,amstext,amsbsy, amsopn,amsthm}
\usepackage{enumerate}
\usepackage{url}
\usepackage{mathtools}
\usepackage{bookmark}
\usepackage{euscript}
\usepackage{mathptmx}
\usepackage{helvet}
\usepackage{courier}
\usepackage{type1cm}
\usepackage{multicol}
\usepackage[bottom]{footmisc}

\newtheorem{theorem}{Theorem}[section]
\newtheorem*{theorem*}{Theorem B}
\newtheorem{lemma}[theorem]{Lemma}

\newtheorem{proposition}[theorem]{Proposition}
\newtheorem{corollary}[theorem]{Corollary}

\newtheorem*{definition*}{Definition}
\newtheorem*{remark*}{Remark}

\newtheorem*{observation*}{Observation}

\newtheorem*{assumption*}{Assumption}
\newtheorem*{question*}{Question}
\newtheorem*{problem*}{Problem}

\newtheorem{claim}[theorem]{Claim}
\newtheorem{remark}{Remark}
\newtheorem{question}{Question}

\usepackage[title]{appendix}

\renewcommand{\theequation}{\thesection.\arabic{equation}}

\begin{document}
\title[Darboux-type formula for Jacobi biorthogonal polynomials]{Darboux-type formula for Jacobi biorthogonal polynomials}

\author
{Zhaofeng Lin}
\address
{Zhaofeng Lin: School of Fundamental Physics and Mathematical Sciences, HIAS, University of Chinese Academy of Sciences, Hangzhou, 310024, China}
\email{linzhaofeng@ucas.ac.cn}

\author
{Kai Wang}
\address
{Kai Wang: School of Mathematical Sciences, Fudan University, Shanghai, 200433, China}
\email{kwang@fudan.edu.cn}

\author
{Zhanhang Zheng}
\address
{Zhanhang Zheng: School of Mathematical Sciences, Fudan University, Shanghai, 200433, China}
\email{zhzheng22@m.fudan.edu.cn}

\thanks{This work is supported by the National Key R\&D Program of China (2024YFA1013400) and the National Natural Science Foundation of China (No.12231005).}

\begin{abstract}
In this paper, we study the asymptotic behavior of Jacobi biorthogonal polynomials. A Darboux-type formula is established using the method of steepest descent. In the proof, we construct an appropriate contour to apply the Rodrigues formula. Our result reduces to the classical Darboux formula in the orthogonal case.
\end{abstract}



\maketitle

\setcounter{equation}{0}

\section{Introduction}
The notion of orthogonal polynomials was extended to biorthogonal polynomials by Konhauser \cite{Kon1, Kon2}. The asymptotics of biorthogonal polynomials have been extensively studied. Lubinsky and Sidi \cite{LSi} derived strong asymptotics for biorthogonal polynomials with respect to powers of $\log x$ by the method of steepest descent. Borrego-Morell and Rafaeli \cite{BR} investigated the uniform asymptotics for a class of biorthogonal polynomials on the unit circle. In \cite{WZ}, Wang and Zhang obtained the asymptotics for certain Laguerre biorthogonal polynomials via the Riemann–Hilbert approach.

The aim of this paper is to establish a Darboux-type formula for Jacobi biorthogonal polynomials using Rodrigues formula in combination with the method of steepest descent. In this section, we first recall Rodrigues formula for Jacobi orthogonal polynomials and the well-known Darboux formula that characterizes their asymptotic behavior. Next, we review the definition of Jacobi biorthogonal polynomials and their Rodrigues formula established by Lubinsky and Soran \cite{LS}, as well as by Lubinsky and Stahl \cite{LSt}. Finally, we state our main result, namely, Theorem~\ref{thm-Darbo-Formu-biop}, which is a Darboux-type formula that provides the asymptotics of Jacobi biorthogonal polynomials.

\subsection{Jacobi orthogonal polynomials}
Consider the Jacobi weight $\omega^{(a,b)}(x)=(1-x)^a(1+x)^b$ on the interval $[-1,1]$, where the integrability of $\omega^{(a,b)}(x)$ is ensured by requiring $a,b>-1$. The Jacobi polynomial $P_n^{(a,b)}(x)$ of degree $n$ is orthogonal on $[-1,1]$ with respect to the Jacobi weight function $\omega^{(a,b)}(x)$. To be precise, the orthogonality of the polynomial $P_n^{(a,b)}(x)$ means that
\[
\int_{-1}^{1}P_n^{(a,b)}(x)\,x^j\,\omega^{(a,b)}(x)\mathrm{d}x\left\{
\begin{array}{cl}
	=0&\text{ if $j=0,1,\cdots,n-1$}
	\vspace{2mm}
	\\
	\neq0&\text{ if $j=n$}
\end{array}\right..
\]
Equivalently, we can write
\begin{align}\label{def-op}
	\int_{-1}^{1}P_n^{(a,b)}(x)\,(1-x)^j\,\omega^{(a,b)}(x)\mathrm{d}x\left\{
	\begin{array}{cl}
		=0&\text{ if $j=0,1,\cdots,n-1$}
		\vspace{2mm}
		\\
		\neq0&\text{ if $j=n$}
	\end{array}\right..
\end{align}

The Jacobi orthogonal polynomial $P_n^{(a,b)}(x)$ is unique up to a non-zero constant factor.  Following \cite[Formula~(4.1.1)]{Sze}, the normalization of $P_n^{(a,b)}(x)$ can be chosen as
\[
P_n^{(a,b)}(1)=\frac{\Gamma(n+a+1)}{\Gamma(n+1)\Gamma(a+1)}.
\]
The polynomial $P_n^{(a,b)}(x)$ can be expressed in various explicit forms, such as the following important representation (see, e.g., \cite[Formula~(4.3.2)]{Sze}):
\begin{align}\label{repre-op}
	P_n^{(a,b)}(x)=\sum_{r=0}^{n}\frac{(-1)^r\,\Gamma(n+a+1)\Gamma(n+b+1)}{\Gamma(n-r+1)\Gamma(r+a+1)\Gamma(r+1)\Gamma(n-r+b+1)}\Big(\frac{1-x}{2}\Big)^{r}\Big(\frac{1+x}{2}\Big)^{n-r}.
\end{align}

The Rodrigues formula for the Jacobi orthogonal polynomial $P_n^{(a,b)}(x)$ is
\begin{align}\label{Rod-op}
	P_n^{(a,b)}(x)=\frac{(-1)^n}{2^nn!}(1-x)^{-a}(1+x)^{-b}\Big(\frac{\mathrm{d}}{\mathrm{d}x}\Big)^n\big\{(1-x)^{n+a}(1+x)^{n+b}\big\}.
\end{align}
Hence by Cauchy's integral formula, we obtain the following form:
\begin{align}\label{Rod-int-op}
P_n^{(a,b)}(x)=\frac{1}{2\pi i}\oint_{C}\Big(\frac{1}{2}\frac{\xi^2-1}{\xi-x}\Big)^n\Big(\frac{1-\xi}{1-x}\Big)^a\Big(\frac{1+\xi}{1+x}\Big)^b\frac{\mathrm{d}\xi}{\xi-x}.
\end{align}
Here $x\neq\pm1$, and the integration is extended in the positive sense around a closed contour $C$ enclosing $\xi=x$, but not the points $\xi=\pm1$. The reader is referred to \cite[Formulas~(4.3.1) and (4.6.1)]{Sze} for the above Rodrigues formula.

Replacing $x$ with $\cos\theta$ in the integral form \eqref{Rod-int-op} of the Rodrigues formula, we obtain
\[
P_n^{(a,b)}(\cos\theta)=\frac{1}{2\pi i}\oint_{C}\Big(\frac{1}{2}\frac{\xi^2-1}{\xi-\cos\theta}\Big)^n\Big(\frac{1-\xi}{1-\cos\theta}\Big)^a\Big(\frac{1+\xi}{1+\cos\theta}\Big)^b\frac{\mathrm{d}\xi}{\xi-\cos\theta}.
\]
Then, using the method of steepest descent, one can show the well-known Darboux formula (see, e.g., \cite[Theorem~8.21.8]{Sze}), which gives the uniform asymptotic estimate for the Jacobi orthogonal polynomial $P_n^{(a,b)}$ as $n\to\infty$:
\begin{align}\label{Darbo-Formu}
	P_n^{(a,b)}(\cos\theta)=\pi^{-\frac{1}{2}}n^{-\frac{1}{2}}\Big(\sin\frac{\theta}{2}\Big)^{-a-\frac{1}{2}}\Big(\cos\frac{\theta}{2}\Big)^{-b-\frac{1}{2}}\cos\Big(N\theta-\frac{a\pi}{2}-\frac{\pi}{4}\Big)+O(n^{-\frac{3}{2}})
\end{align}
with $N=n+\frac{a+b+1}{2}$. Here the bound of the error term holds uniformly for $\theta\in[\varepsilon,\pi-\varepsilon]$ with any fixed $0<\varepsilon<\frac{\pi}{2}$. Note that $\cos\theta$ strictly decreases from $1$ to $-1$ as $\theta$ varies from $0$ to $\pi$. Darboux formula actually gives a uniform estimate for $P_n^{(a,b)}(x)$ on any closed subinterval of $(-1,1)$.

\subsection{Jacobi biorthogonal polynomials}
Fix $\alpha>0$. Consider the Jacobi weight $\omega^{(a,b)}(x)=(1-x)^a(1+x)^b$ with $a,b>-1$ on the interval $[-1,1]$.  The Jacobi biorthogonal polynomial $P_n^{(\alpha,a,b)}(x)$ of degree $n$ satisfies the following biorthogonality relation \eqref{def-biop}, which extends the usual orthogonality relation \eqref{def-op}:
\begin{align}\label{def-biop}
	\int_{-1}^{1}P_n^{(\alpha,a,b)}(x)\,(1-x)^{\alpha j}\,\omega^{(a,b)}(x)\mathrm{d}x\left\{
	\begin{array}{cl}
		=0&\text{ if $j=0,1,\cdots,n-1$}
		\vspace{2mm}
		\\
		\neq0&\text{ if $j=n$}
	\end{array}\right..
\end{align}

The existence and uniqueness (up to a non-zero constant factor) of the Jacobi biorthogonal polynomial $P_n^{(\alpha,a,b)}(x)$ are guaranteed by Lubinsky and Soran \cite{LS}. We can choose the normalization of $P_n^{(\alpha,a,b)}(x)$ as
\[
P_n^{(\alpha,a,b)}(1)=\frac{\Gamma(n+\frac{a+1}{\alpha})}{\Gamma(n+1)\Gamma(\frac{a+1}{\alpha})}.
\]
Then, according to Madhekar and Thakare \cite[Formula~(11)]{MT}, the Jacobi biorthogonal polynomial $P_n^{(\alpha,a,b)}(x)$ has the following representation:
\begin{align}\label{repre-biop}
		P_n^{(\alpha,a,b)}(x)=\sum_{r=0}^{n}\sum_{s=0}^{r}\frac{(-1)^s\,\Gamma(n+\frac{s+a+1}{\alpha})\Gamma(n+b+1)}{\Gamma(n+1)\Gamma(s+1)\Gamma(r-s+1)\Gamma(\frac{s+a+1}{\alpha})\Gamma(n-r+b+1)}\Big(\frac{1-x}{2}\Big)^{r}\Big(\frac{1+x}{2}\Big)^{n-r}.
\end{align}
When $\alpha=1$, by careful calculation, one can derive
\begin{align}\label{cal-gam}
	\sum_{s=0}^{r}\frac{(-1)^s\,\Gamma(n+s+a+1)}{\Gamma(n+1)\Gamma(s+1)\Gamma(r-s+1)\Gamma(s+a+1)}=\frac{(-1)^r\,\Gamma(n+a+1)}{\Gamma(n-r+1)\Gamma(r+a+1)\Gamma(r+1)}.
\end{align}
We leave the calculation of identity \eqref{cal-gam} to Appendix~\ref{prf-gam}. Thus by \eqref{cal-gam}, the biorthogonal expression \eqref{repre-biop} reduces to the orthogonal case \eqref{repre-op}, namely,
\[
P_n^{(\alpha=1,a,b)}(x)=P_n^{(a,b)}(x).
\]

The Rodrigues formula for the Jacobi biorthogonal polynomial $P_n^{(\alpha, a,b)}(x)$ was established by Lubinsky and Soran \cite[Theorem~1]{LS}, and by Lubinsky and Stahl \cite[Theorem~1]{LSt} for the special case $a=b=0$. In \cite{LS} and \cite{LSt}, the authors consider Rodrigues formula for the Jacobi biorthogonal polynomials on the unit interval $[0,1]$ with weight function $x^a(1-x)^b$. Then, using the transformation $x\mapsto\frac{1-x}{2}$ to convert biorthogonal polynomials on $[0,1]$ to those on $[-1,1]$ and vice versa, we obtain the following Rodrigues formula for $P_n^{(\alpha, a,b)}(x)$:
\begin{align}\label{Rod-biop}
	\begin{split}
		P_n^{(\alpha,a,b)}(x)&=\frac{(-1)^n2^n}{n!}\Big(\frac{1-x}{2}\Big)^{\alpha-a-1}\Big(\frac{1+x}{2}\Big)^{-b}\\
		&\quad\times\Big[\Big(\frac{\mathrm{d}}{\mathrm{d}\xi}\Big)^n\Big\{\Big(\frac{1-\xi}{2}\Big)^{n+\frac{a+1}{\alpha}-1}\Big[1-\Big(\frac{1-\xi}{2}\Big)^{\frac{1}{\alpha}}\Big]^{n+b}\Big\}\Big]_{\xi=1-2(\frac{1-x}{2})^{\alpha}}.
	\end{split}
\end{align}
Then, by Cauchy's integral formula, we can write
\begin{align}\label{Rod-int-biop}
	\begin{split}
		P_n^{(\alpha,a,b)}(x)&=\frac{1}{2\pi i}\oint_{C}\Big\{\frac{(\xi-1)[1-(\frac{1-\xi}{2})^{\frac{1}{\alpha}}]}{\xi-[1-2(\frac{1-x}{2})^{\alpha}]}\Big\}^n\Big(\frac{1-\xi}{2}\Big)^{\frac{a+1}{\alpha}-1}\Big(\frac{1-x}{2}\Big)^{\alpha-a-1}\\
		&\qquad\qquad\times\Big[1-\Big(\frac{1-\xi}{2}\Big)^{\frac{1}{\alpha}}\Big]^b\Big(\frac{1+x}{2}\Big)^{-b}\frac{\mathrm{d}\xi}{\xi-[1-2(\frac{1-x}{2})^{\alpha}]}.
	\end{split}
\end{align}
Here $x\neq\pm1$, and the integration is extended in the positive sense around a closed contour $C$ enclosing $\xi=1-2(\frac{1-x}{2})^{\alpha}$, but not the points $\xi=\pm1$. In particular, in the orthogonal case $\alpha=1$, the formulas \eqref{Rod-biop} and \eqref{Rod-int-biop} reduce to the formulas \eqref{Rod-op} and \eqref{Rod-int-op} respectively.

\subsection{Main result: Darboux-type formula}
Our main goal in this paper is to establish the Darboux-type formula for the Jacobi biorthogonal polynomial $P_n^{(\alpha, a,b)}$, analogous to \eqref{Darbo-Formu} for the Jacobi orthogonal polynomial $P_n^{(a,b)}$.

Replacing $x$ with $x(\theta)$ in the integral form \eqref{Rod-int-biop} of the Rodrigues formula, we obtain
\begin{align}\label{Rod-int-biop-tha}
	\begin{split}
		P_n^{(\alpha,a,b)}\big(x(\theta)\big)&=\frac{1}{2\pi i}\oint_{C}\Big\{\frac{(\xi-1)[1-(\frac{1-\xi}{2})^{\frac{1}{\alpha}}]}{\xi-[1-2(\frac{1-x(\theta)}{2})^{\alpha}]}\Big\}^n\Big(\frac{1-\xi}{2}\Big)^{\frac{a+1}{\alpha}-1}\Big(\frac{1-x(\theta)}{2}\Big)^{\alpha-a-1}\\
		&\qquad\qquad\times\Big[1-\Big(\frac{1-\xi}{2}\Big)^{\frac{1}{\alpha}}\Big]^b\Big(\frac{1+x(\theta)}{2}\Big)^{-b}\frac{\mathrm{d}\xi}{\xi-[1-2(\frac{1-x(\theta)}{2})^{\alpha}]}.
	\end{split}
\end{align}
We aim to find suitable $x(\theta)$ and the closed contour $C$ such that the method of steepest descent can be applied to expression \eqref{Rod-int-biop-tha}. Furthermore, $x(\theta)$ is required to cover $(-1,1)$ as $\theta$ varies over $(0,\pi)$. This would allow us to obtain a uniform asymptotic estimate on any closed subinterval of $(-1,1)$ for the Jacobi biorthogonal polynomial $P_n^{(\alpha, a,b)}(x)$.

\subsubsection{Choice of $x(\theta)$}
We now begin by defining $x(\theta)$. As for the construction of the closed contour $C$, it will be given in Section~\ref{Constru-clocon}. For any $\alpha>0$ and $\theta\in(0,\pi)$, let
\begin{align}\label{def-Tha}
	\Theta_{\alpha}(\theta):=\frac{\sin\theta}{(1+\alpha)\sin\frac{\pi-\theta}{1+\alpha}},\quad\Theta_{\frac{1}{\alpha}}(\theta):=\frac{\sin\theta}{(1+\frac{1}{\alpha})\sin\frac{\pi-\theta}{1+\frac{1}{\alpha}}}.
\end{align}
Define $x(\theta)=x_{\alpha}(\theta)$ as follows:
\begin{align}\label{def-xtha}
	x(\theta):=1-2\Theta_{\frac{1}{\alpha}}(\theta)\big[\Theta_{\alpha}(\theta)\big]^{\frac{1}{\alpha}}=1-2\Big[\frac{\sin\theta}{(1+\frac{1}{\alpha})\sin\frac{\pi-\theta}{1+\frac{1}{\alpha}}}\Big]\Big[\frac{\sin\theta}{(1+\alpha)\sin\frac{\pi-\theta}{1+\alpha}}\Big]^{\frac{1}{\alpha}}.
\end{align}
It is easy to verify that when $\alpha=1$, our $x(\theta)$ reduces to $\cos\theta$ in the orthogonal case. 

\begin{lemma}\label{lem-xtha}
	For any fixed $\alpha>0$, the function $x(\theta)$ is strictly decreasing on $(0,\pi)$ with $x(0^+)=1$ and $x(\pi^-)=-1$.
\end{lemma}

\begin{proof}
	Recalling the definition of $\Theta_{\alpha}(\theta)$ in \eqref{def-Tha}, one can show that
	\[
	\Theta_{\alpha}'(\theta)=\frac{\mathrm{U}_{\alpha}(\theta)}{(1+\alpha)^2\sin^2\frac{\pi-\theta}{1+\alpha}}\quad\text{with}\quad\mathrm{U}_{\alpha}(\theta):=(1+\alpha)\cos\theta\sin\frac{\pi-\theta}{1+\alpha}+\sin\theta\cos\frac{\pi-\theta}{1+\alpha}.
	\]
	Note that
	\[
	\mathrm{U}_{\alpha}'(\theta)=-\frac{\alpha(2+\alpha)}{1+\alpha}\sin\theta\sin\frac{\pi-\theta}{1+\alpha}<0,\quad \theta\in(0,\pi).
	\]
	Hence $\mathrm{U}_{\alpha}(\theta)$ is strictly decreasing on $(0,\pi)$ with $\mathrm{U}_{\alpha}(\pi^-)=0$, and consequently $\Theta_{\alpha}'(\theta)>0$ for $\theta\in(0,\pi)$. Thus $\Theta_{\alpha}(\theta)$ is strictly increasing on $(0,\pi)$ with $\Theta_{\alpha}(0^+)=0$ and $\Theta_{\alpha}(\pi^-)=1$. Similarly, $\Theta_{\frac{1}{\alpha}}(\theta)$ is also strictly increasing on $(0,\pi)$ with $\Theta_{\frac{1}{\alpha}}(0^+)=0$ and $\Theta_{\frac{1}{\alpha}}(\pi^-)=1$. Finally, from the expression of $x(\theta)$ in \eqref{def-xtha}, we conclude that $x(\theta)$ is strictly decreasing on $(0,\pi)$ with $x(0^+)=1$ and $x(\pi^-)=-1$.
\end{proof}

\subsubsection{Darboux-type formula}
It follows from Lemma~\ref{lem-xtha} that in the general biorthogonal case $\alpha>0$, our choice of $x(\theta)$ in the form of \eqref{def-xtha} can cover $(-1,1)$ as $\theta$ varies over $(0,\pi)$. We now proceed to establish the Darboux-type formula for the Jacobi biorthogonal polynomial $P_n^{(\alpha, a,b)}\big(x(\theta)\big)$, which provides a uniform asymptotic estimate on any closed subinterval of $(-1,1)$ for the Jacobi biorthogonal polynomial $P_n^{(\alpha, a,b)}(x)$.

Recall the choice of  $x(\theta)$ defined in \eqref{def-xtha}:
\[
x(\theta)=x_{\alpha}(\theta)=1-2\Theta_{\frac{1}{\alpha}}(\theta)\big[\Theta_{\alpha}(\theta)\big]^{\frac{1}{\alpha}},
\]
where $\Theta_{\alpha}(\theta)$ and $\Theta_{\frac{1}{\alpha}}(\theta)$ are defined in \eqref{def-Tha}:
\[
\Theta_{\alpha}(\theta)=\frac{\sin\theta}{(1+\alpha)\sin\frac{\pi-\theta}{1+\alpha}},\quad\Theta_{\frac{1}{\alpha}}(\theta)=\frac{\sin\theta}{(1+\frac{1}{\alpha})\sin\frac{\pi-\theta}{1+\frac{1}{\alpha}}}.
\]
Define a function $M_{\alpha}(\theta)$ independent of $n$  as follows:
\begin{align}\label{def-Balp}
	M_{\alpha}(\theta):=\frac{e^{-i[\frac{\pi}{2}+\frac{\pi-\theta}{1+\alpha}(a+b+1)]}\cdot[e^{i\frac{\pi-\theta}{1+\alpha}}-\Theta_{\frac{1}{\alpha}}(\theta)]^{b+\frac{1}{2}}\cdot[e^{i\frac{\pi-\theta}{1+\frac{1}{\alpha}}}-\Theta_{\alpha}(\theta)]}{[\Theta_{\frac{1}{\alpha}}(\theta)]^{\frac{1}{2}}\,[\Theta_{\alpha}(\theta)]^{\frac{a+1}{\alpha}-1}\cdot\{1-\Theta_{\frac{1}{\alpha}}(\theta)[\Theta_{\alpha}(\theta)]^{\frac{1}{\alpha}}\}^b\cdot\{1+[\Theta_{\alpha}(\theta)]^2-2\Theta_{\alpha}(\theta)\cos\frac{\pi-\theta}{1+\frac{1}{\alpha}}\}}.
\end{align}

\begin{theorem}\label{thm-Darbo-Formu-biop}
	For any fixed $\alpha\geq1$ and $a,b>-1$, as $n\to\infty$, we have
	\[
	P_n^{(\alpha, a,b)}\big(x(\theta)\big)=\frac{\sqrt{2}\,\alpha}{\sqrt{1+\alpha}}\pi^{-\frac{1}{2}}n^{-\frac{1}{2}}\Bigg(\frac{\sin\frac{\pi-\theta}{1+\alpha}}{\sin\frac{\pi-\theta}{1+\frac{1}{\alpha}}}\Bigg)^n\Big(\mathrm{Re}\big\{M_{\alpha}(\theta)e^{in\theta}\big\}\Big)\big[1+O(n^{-1})\big],
	\]
	where the bound of the error term holds uniformly for $\theta\in[\varepsilon,\pi-\varepsilon]$ with any fixed $0<\varepsilon<\frac{\pi}{2}$.
\end{theorem}

We will prove Theorem~\ref{thm-Darbo-Formu-biop} using the method of steepest descent in Section~\ref{prf-thm-metstedes}. It is worth mentioning that when $\alpha=1$, a simple calculation shows that
\[
x(\theta)=\cos\theta\quad\text{and}\quad M_{\alpha=1}(\theta)=e^{i(\frac{a+b+1}{2}\theta-\frac{a\pi}{2}-\frac{\pi}{4})}\Big(\sin\frac{\theta}{2}\Big)^{-a-\frac{1}{2}}\Big(\cos\frac{\theta}{2}\Big)^{-b-\frac{1}{2}}.
\]
Then the conclusion in Theorem~\ref{thm-Darbo-Formu-biop} reduces exactly to the well-known Darboux formula \eqref{Darbo-Formu} in the orthogonal case, with $N=n+\frac{a+b+1}{2}$:
\begin{align*}
	P_n^{(\alpha=1,a,b)}(\cos\theta)=\pi^{-\frac{1}{2}}n^{-\frac{1}{2}}\Big(\sin\frac{\theta}{2}\Big)^{-a-\frac{1}{2}}\Big(\cos\frac{\theta}{2}\Big)^{-b-\frac{1}{2}}\cos\Big(N\theta-\frac{a\pi}{2}-\frac{\pi}{4}\Big)+O(n^{-\frac{3}{2}}).
\end{align*}

\begin{remark}\label{rem-1}
	Although the parameter $\alpha$ of Jacobi biorthogonal polynomials ranges over $\alpha > 0$, the requirement for the parameter in Theorem~\ref{thm-Darbo-Formu-biop}, which is obtained using the method of steepest descent, is $\alpha \geq 1$. This is because the monotonicity condition for the method of steepest descent can only be verified when $\alpha \geq 1$; that is, Lemma~\ref{lem-mono-T} below requires $\alpha \geq 1$. As stated in Remark~\ref{rem-2} following Lemma~\ref{lem-mono-T}, a more essential reason is that the proof of Lemma~\ref{lem-mono-T} requires Claim~\ref{clm-anaT'}, which holds only for $\alpha \geq 1$. In Remark~\ref{rem-3} following Claim~\ref{clm-anaT'}, we will show that Claim~\ref{clm-anaT'} is false for $0 < \alpha < 1$, which also indicates that the method of steepest descent fails for the case $0 < \alpha < 1$.
\end{remark}

\begin{question}
	Although the method of steepest descent fails for $0<\alpha<1$, it may still be possible to derive a Darboux-type formula for Jacobi biorthogonal polynomials by other methods. If so, would it take the same form as in Theorem~\ref{thm-Darbo-Formu-biop}?
\end{question}

The following uniform estimate for the upper bound of $P_n^{(\alpha, a,b)}(x)$ on any closed subinterval of $(-1,1)$, or equivalently the uniform estimate for the upper bound of $P_n^{(\alpha, a,b)}(x(\theta))$ on any closed subinterval of $(0,\pi)$, is a direct corollary of Theorem~\ref{thm-Darbo-Formu-biop}.

\begin{corollary}\label{cor-upperbound-biop}
	For any fixed $\alpha\geq1$ and $a,b>-1$, and any fixed $0<\varepsilon<\frac{\pi}{2}$, there exists a constant $\mathrm{Const}>0$ such that for all $\theta\in[\varepsilon,\pi-\varepsilon]$ and all $n\geq0$,
	\[
	\big|P_n^{(\alpha, a,b)}\big(x(\theta)\big)\big|\leq\mathrm{Const}\cdot n^{-\frac{1}{2}}\Bigg(\frac{\sin\frac{\pi-\theta}{1+\alpha}}{\sin\frac{\pi-\theta}{1+\frac{1}{\alpha}}}\Bigg)^n.
	\]
\end{corollary}

\begin{proof}
	It suffices to note that the denominator of $M_{\alpha}(\theta)$ has a positive lower bound on $[\varepsilon,\pi-\varepsilon]$ and hence $M_{\alpha}(\theta)$ is bounded on $[\varepsilon,\pi-\varepsilon]$. Corollary~\ref{cor-upperbound-biop} then follows directly from Theorem~\ref{thm-Darbo-Formu-biop}.
\end{proof}

\section{The method of steepest descent}\label{prf-thm-metstedes}
In this section, we prove Theorem~\ref{thm-Darbo-Formu-biop} using the method of steepest descent. For this purpose, we first construct the closed contour $C$ appearing in the integral \eqref{Rod-int-biop-tha} derived from the Rodrigues formula. We then analyze the conditions required for the method of steepest descent to be applicable. Finally, we perform the analysis and estimation of the path integrals, thereby obtaining a Darboux-type formula that characterizes the asymptotic behavior of Jacobi biorthogonal polynomials.

\subsection{Construction  of the closed contour $C$}\label{Constru-clocon}
We first give the construction of the closed contour $C=C_{\alpha}$ appearing in expression \eqref{Rod-int-biop-tha} above. Let us consider the closed contour $C$ given by the following parametrization $\xi(\varphi)=\xi_{\alpha}(\varphi)$ with $\varphi\in[-\pi,\pi)$:
\begin{align}\label{parame-xiphi}
	\xi(\varphi):=\left\{
	\begin{array}{cl}
		1-2\Big[\frac{\sin\varphi}{(1+\frac{1}{\alpha})\sin\frac{\pi-\varphi}{1+\frac{1}{\alpha}}}\Big]^{\alpha}\exp\Big(-i\frac{\pi-\varphi}{1+\frac{1}{\alpha}}\Big)&\text{ if $0\leq\varphi<\pi$}
		\vspace{2mm}
		\\
		1-2\Big[\frac{-\sin\varphi}{(1+\frac{1}{\alpha})\sin\frac{\pi+\varphi}{1+\frac{1}{\alpha}}}\Big]^{\alpha}\exp\Big(i\frac{\pi+\varphi}{1+\frac{1}{\alpha}}\Big)\quad&\text{ if $-\pi\leq\varphi<0$}
	\end{array}\right..
\end{align}

The following Figure~\ref{clocon} shows the images of the closed contour obtained from expression \eqref{parame-xiphi} for the parameter $\alpha$ taking values $1/2$, $1$, $2$ and $4$. Although, as stated in Remark~\ref{rem-1} following Theorem~\ref{thm-Darbo-Formu-biop}, we only consider the case $\alpha\geq1$, the figure for $\alpha=1/2$ is also included for illustration.

\begin{figure}[htbp]
	\centering
	\begin{minipage}{0.23\textwidth}
		\centering
		\includegraphics[width=\linewidth]{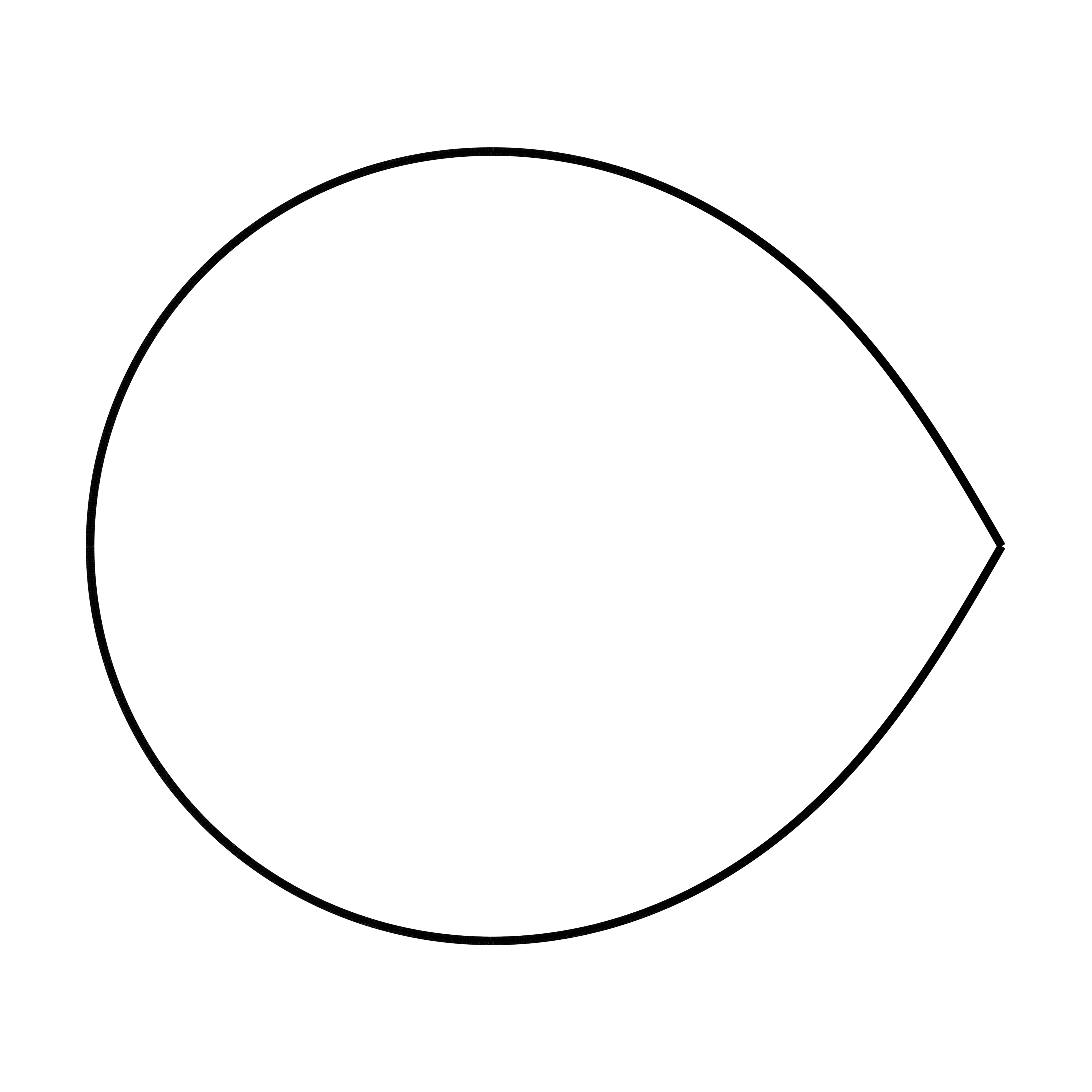}
		\text{$\alpha=1/2$}
	\end{minipage}
	\hfill
	\begin{minipage}{0.23\textwidth}
		\centering
		\includegraphics[width=\linewidth]{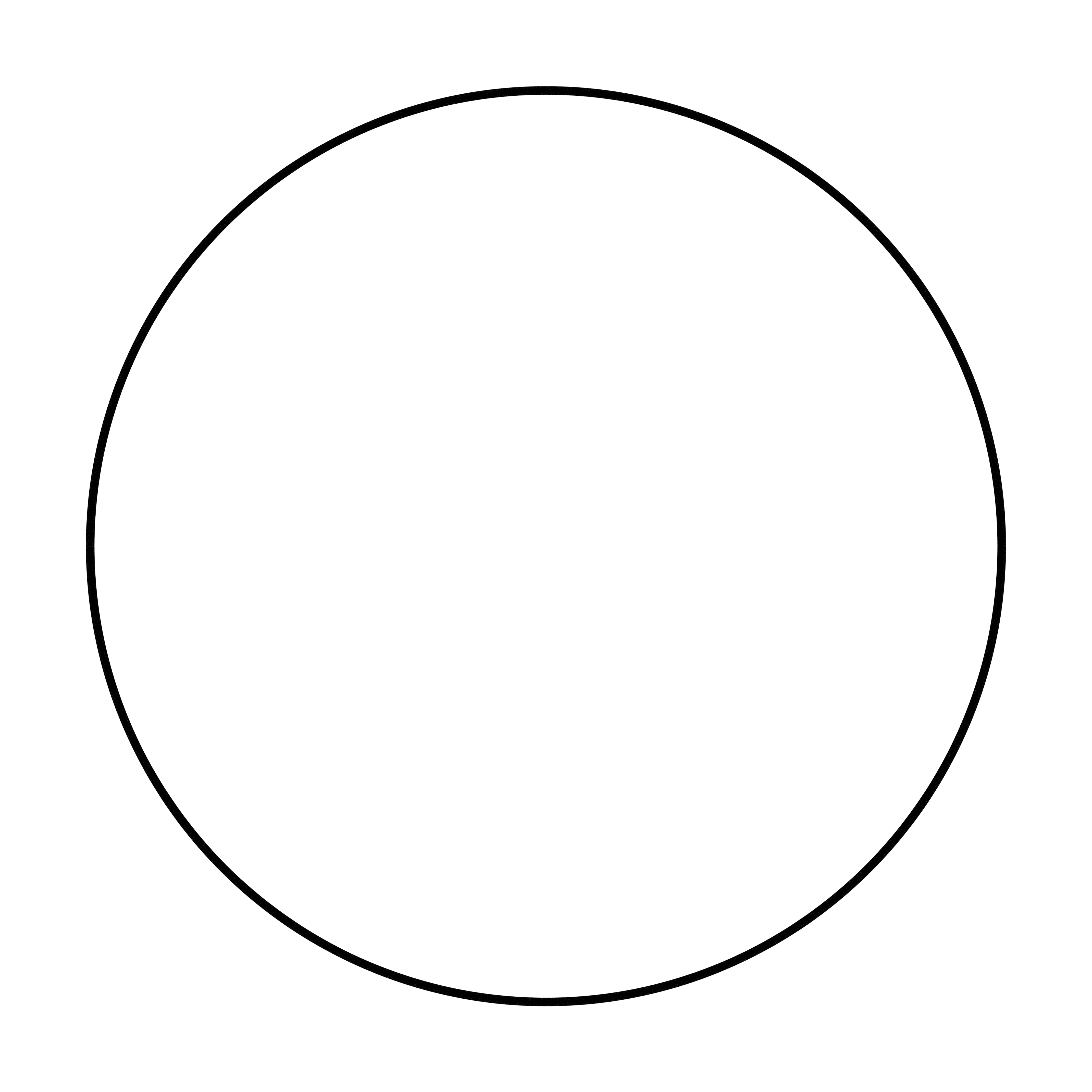}
		\text{$\alpha=1$}
	\end{minipage}
	\hfill
	\begin{minipage}{0.23\textwidth}
		\centering
		\includegraphics[width=\linewidth]{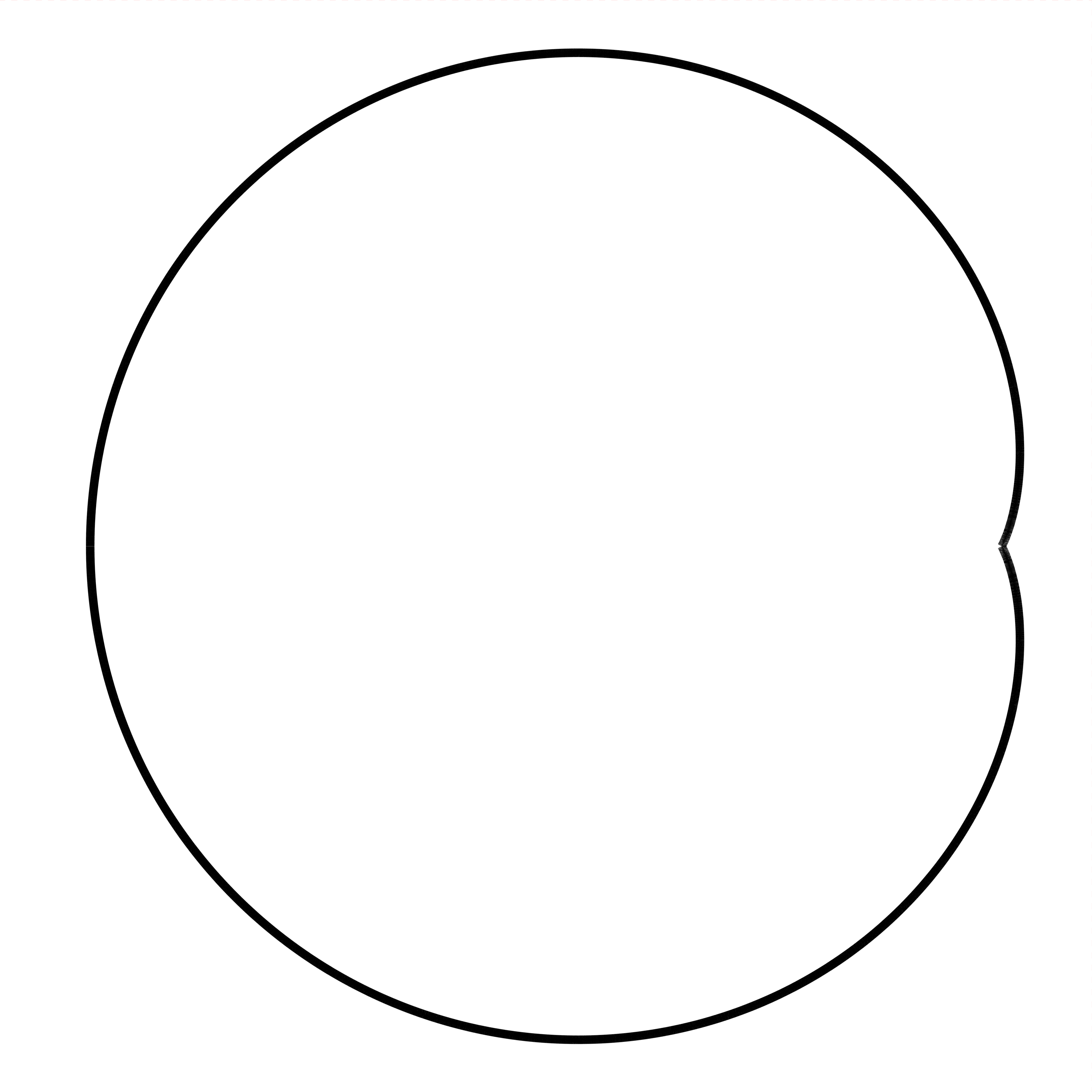}
		\text{$\alpha=2$}
	\end{minipage}
	\hfill
	\begin{minipage}{0.23\textwidth}
		\centering
		\includegraphics[width=\linewidth]{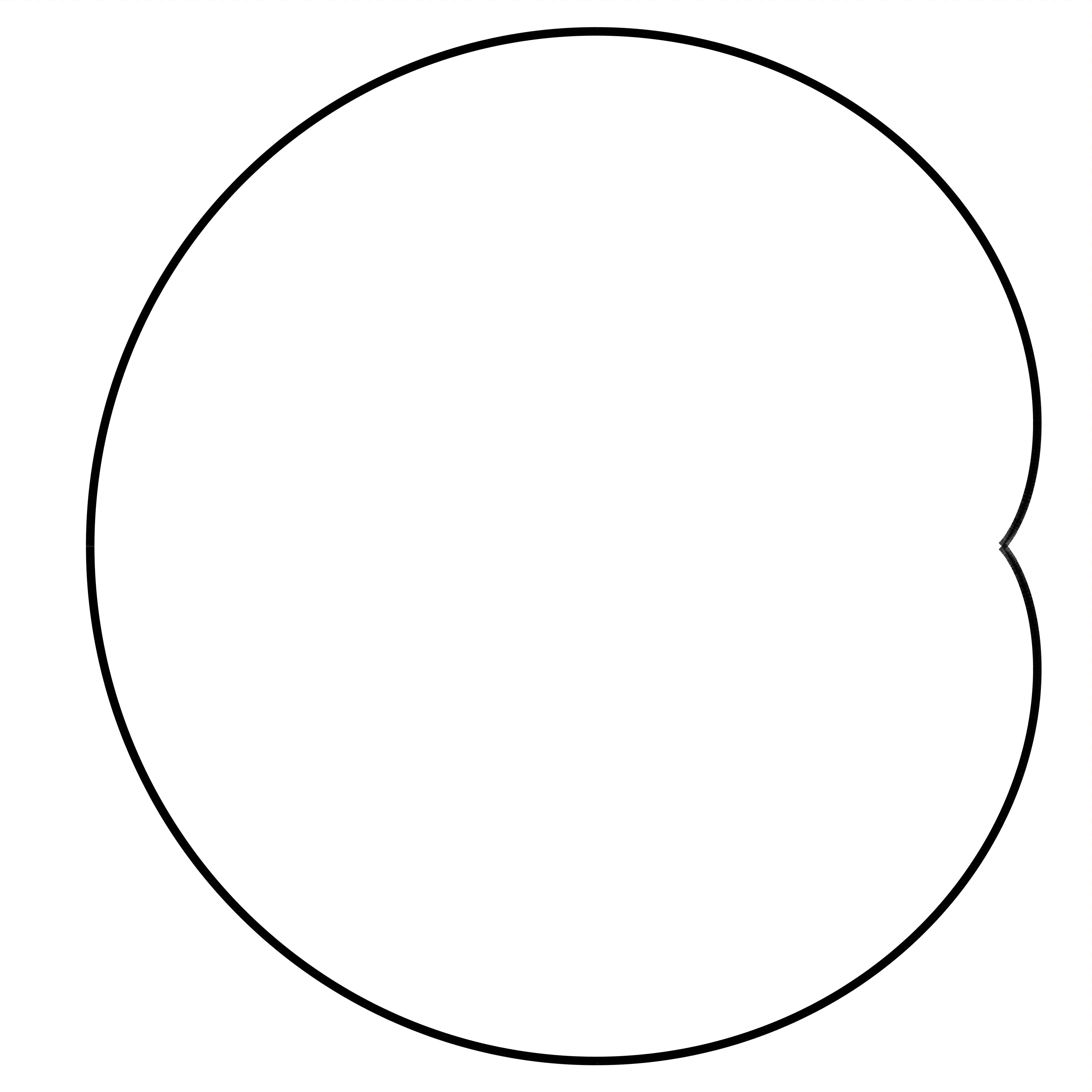}
		\text{$\alpha=4$}
	\end{minipage}
	\caption{The closed contour $C=C_\alpha$ for $\alpha=1/2$, $1$, $2$ and $4$ respectively.}\label{clocon}
\end{figure}

Reconsidering $P_n^{(\alpha,a,b)}(x)$ in the expressions \eqref{Rod-biop} and \eqref{Rod-int-biop}, since the function
\[
\Big(\frac{1-\xi}{2}\Big)^{n+\frac{a+1}{\alpha}-1}\Big[1-\Big(\frac{1-\xi}{2}\Big)^{\frac{1}{\alpha}}\Big]^{n+b}
\]
may have branch points at $\pm1$, the closed contour $C$ in Figure~\ref{clocon}  actually needs to be modified to $C(\tau)$ by a small $\tau$-dependent deformation near the branch points $\pm1$ in the right panel of the following Figure~\ref{clocon-modif} (with $\alpha=2$ as an example). 

\begin{figure}[htbp]
	\centering
	\begin{minipage}{0.32\textwidth}
		\centering
		\includegraphics[width=\linewidth]{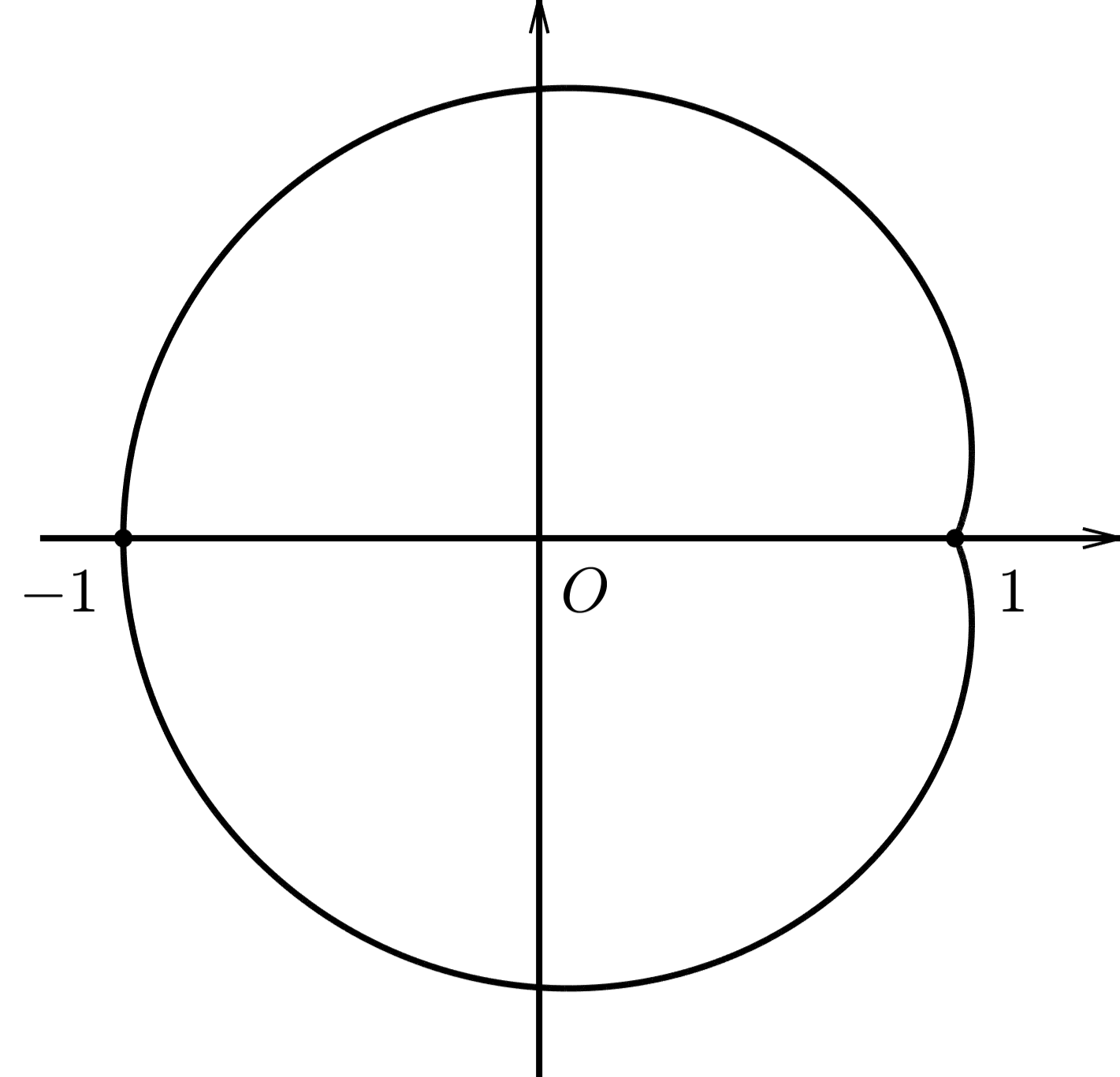}
		\text{The closed contour $C$.}
	\end{minipage}
	\begin{minipage}{0.175\textwidth}
		\centering
		\text{$\xrightarrow{\text{\,\,\,\,modify\,\,\,\,}}$}
		
		\text{}
	\end{minipage}
	\begin{minipage}{0.32\textwidth}
		\centering
		\includegraphics[width=\linewidth]{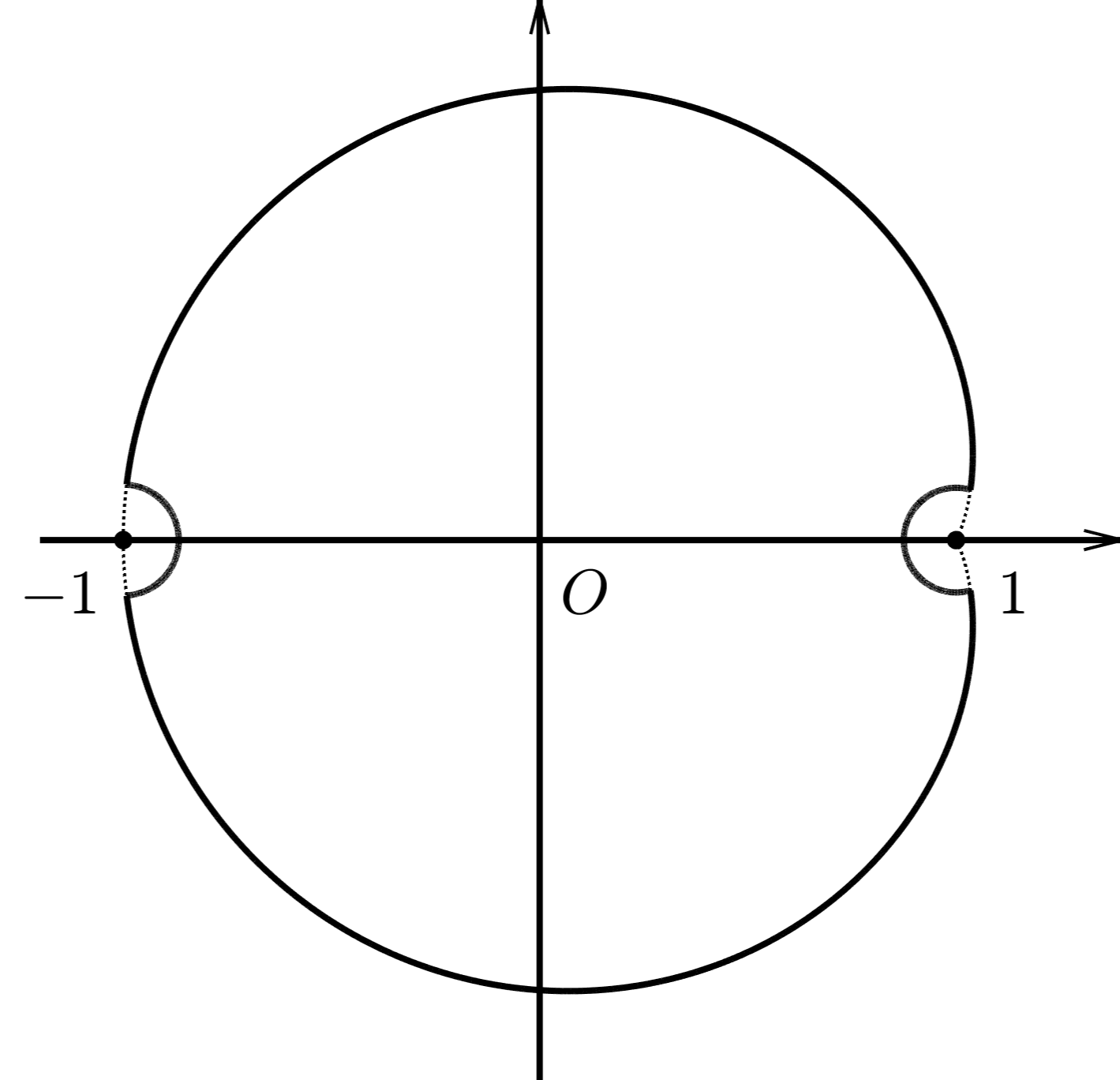}
		\text{The modified closed contour $C(\tau)$.}
	\end{minipage}
	\caption{The modification of $C$ into $C(\tau)$.}
	\label{clocon-modif}
\end{figure}

Note that the closed contour $C$ derived from \eqref{parame-xiphi} passes through the points $\pm1$, whereas the modified closed contour $C(\tau)$ bypasses them. Consequently, by Cauchy's integral formula, the formula that can be directly applied to \eqref{Rod-biop} is the contour integral along $C(\tau)$:
\[
P_n^{(\alpha,a,b)}(x)=\frac{1}{2\pi i}\oint_{C(\tau)}\cdots\,\,\mathrm{d}\xi,
\]
where the integrand is the same as that in formula \eqref{Rod-int-biop}.

However, we can prove that when $n\geq\max\{1-\frac{a+1}{\alpha},-b\}$, the following limit holds:
\begin{align}\label{cancl-modifi}
	\lim_{\tau\to0^+}\frac{1}{2\pi i}\oint_{C(\tau)}\cdots\,\,\mathrm{d}\xi\,=\,\frac{1}{2\pi i}\oint_{C}\cdots\,\,\mathrm{d}\xi.
\end{align}
Based on this, we may take the contour in \eqref{Rod-int-biop} and \eqref{Rod-int-biop-tha} as the original closed contour $C$ (left panel of Figure~\ref{clocon-modif}) constructed from \eqref{parame-xiphi}.

\begin{proof}[Proof of \eqref{cancl-modifi}]
	By the Rodrigues formula \eqref{Rod-biop}, Cauchy's integral formula gives
	\[
	P_n^{(\alpha,a,b)}(x)=(-1)^n2^n\Big(\frac{1-x}{2}\Big)^{\alpha-a-1}\Big(\frac{1+x}{2}\Big)^{-b}\frac{1}{2\pi i}\oint_{C(\tau)}\frac{(\frac{1-\xi}{2})^{n+\frac{a+1}{\alpha}-1}[1-(\frac{1-\xi}{2})^{\frac{1}{\alpha}}]^{n+b}}{\{\xi-[1-2(\frac{1-x}{2})^{\alpha}]\}^{n+1}}\mathrm{d}\xi.
	\]
	Let $n\geq\max\{1-\frac{a+1}{\alpha},-b\}$ be fixed. Then it is easy to verify that the function
	\[
	\Big(\frac{1-\xi}{2}\Big)^{n+\frac{a+1}{\alpha}-1}\Big[1-\Big(\frac{1-\xi}{2}\Big)^{\frac{1}{\alpha}}\Big]^{n+b}
	\]
	is bounded when $\xi$ is sufficiently close to $\pm1$. Moreover, for any fixed $x\in(-1,1)$, the function
	\[
	\Big\{\xi-\Big[1-2\Big(\frac{1-x}{2}\Big)^{\alpha}\Big]\Big\}^{n+1}
	\]
	has a positive lower bound when $\xi$ is sufficiently close to $\pm1$. Hence by noticing that
	\[
	\lim_{\tau\to0^+}\mathrm{length}\big(C(\tau)-C\big)=0,
	\]
	we obtain
	\[
	\lim_{\tau\to0^+}\oint_{C(\tau)}\frac{(\frac{1-\xi}{2})^{n+\frac{a+1}{\alpha}-1}[1-(\frac{1-\xi}{2})^{\frac{1}{\alpha}}]^{n+b}}{\{\xi-[1-2(\frac{1-x}{2})^{\alpha}]\}^{n+1}}\mathrm{d}\xi\,=\,\oint_{C}\frac{(\frac{1-\xi}{2})^{n+\frac{a+1}{\alpha}-1}[1-(\frac{1-\xi}{2})^{\frac{1}{\alpha}}]^{n+b}}{\{\xi-[1-2(\frac{1-x}{2})^{\alpha}]\}^{n+1}}\mathrm{d}\xi.
	\]
	Therefore, we can take the contour in \eqref{Rod-int-biop} and \eqref{Rod-int-biop-tha} as the original closed contour $C$.
\end{proof}

\subsection{Integral along the curve $C_+$}
Let $C_+$ denote the part of the closed contour $C$ in the upper half-plane oriented in the positive direction. That is, the curve $C_+$ is determined by the parametrization
\begin{align}\label{parame-xiphi-up}
	\xi(\varphi)=1-2\Big[\frac{\sin\varphi}{(1+\frac{1}{\alpha})\sin\frac{\pi-\varphi}{1+\frac{1}{\alpha}}}\Big]^{\alpha}\exp\Big(-i\frac{\pi-\varphi}{1+\frac{1}{\alpha}}\Big)=1-2\big[\Theta_{\frac{1}{\alpha}}(\varphi)\big]^{\alpha}\exp\Big(-i\frac{\pi-\varphi}{1+\frac{1}{\alpha}}\Big)
\end{align}
with $0\leq\varphi<\pi$. Here we use the notation \eqref{def-Tha}:
\begin{align}\label{def-Tha-phi}
	\Theta_{\frac{1}{\alpha}}(\varphi):=\frac{\sin\varphi}{(1+\frac{1}{\alpha})\sin\frac{\pi-\varphi}{1+\frac{1}{\alpha}}}.
\end{align}

\begin{lemma}\label{cla-xi'}
	Fix $\alpha>0$. Then $\xi'(\varphi)\neq0$ for all $\varphi\in(0,\pi)$.
\end{lemma}

\begin{proof}
	By the expression of $\xi(\varphi)$ in \eqref{parame-xiphi-up}, one can show that
	\[
	\xi'(\varphi)=-\frac{2\alpha}{1+\alpha}\big[\Theta_{\frac{1}{\alpha}}(\varphi)\big]^{\alpha-1}\exp\Big(-i\frac{\pi-\varphi}{1+\frac{1}{\alpha}}\Big)\Big[(1+\alpha)\Theta_{\frac{1}{\alpha}}'(\varphi)+i\,\Theta_{\frac{1}{\alpha}}(\varphi)\Big].
	\]
	It follows from the proof of Lemma~\ref{lem-xtha} that $\Theta_{\frac{1}{\alpha}}(\varphi)\in(0,1)$ for all $\varphi\in(0,\pi)$. Hence $\xi'(\varphi)\neq0$ for all $\varphi\in(0,\pi)$.
\end{proof}

By symmetry of the closed contour $C$ and formula \eqref{Rod-int-biop-tha}, we have the following expression:
\begin{align}\label{Rod-int-biop-tha-Re}
	\begin{split}
		P_n^{(\alpha,a,b)}\big(x(\theta)\big)&=\mathrm{Re}\Bigg\{\frac{1}{\pi i}\int_{C_+}\Big\{\frac{(\xi-1)[1-(\frac{1-\xi}{2})^{\frac{1}{\alpha}}]}{\xi-[1-2(\frac{1-x(\theta)}{2})^{\alpha}]}\Big\}^n\Big(\frac{1-\xi}{2}\Big)^{\frac{a+1}{\alpha}-1}\Big(\frac{1-x(\theta)}{2}\Big)^{\alpha-a-1}\\
		&\qquad\qquad\qquad\times\Big[1-\Big(\frac{1-\xi}{2}\Big)^{\frac{1}{\alpha}}\Big]^b\Big(\frac{1+x(\theta)}{2}\Big)^{-b}\frac{\mathrm{d}\xi}{\xi-[1-2(\frac{1-x(\theta)}{2})^{\alpha}]}\Bigg\}.
	\end{split}
\end{align}
Let us define $t(\theta)=t_{\alpha}(\theta)$ as follows:
\begin{align}\label{def-ttha}
	t(\theta):=1-2\big[\Theta_{\frac{1}{\alpha}}(\theta)\big]^{\alpha}\,\Theta_{\alpha}(\theta)=1-2\Big[\frac{\sin\theta}{(1+\frac{1}{\alpha})\sin\frac{\pi-\theta}{1+\frac{1}{\alpha}}}\Big]^{\alpha}\,\Big[\frac{\sin\theta}{(1+\alpha)\sin\frac{\pi-\theta}{1+\alpha}}\Big].
\end{align}
Then, by \eqref{def-xtha} and \eqref{def-ttha}, it is easy to verify that
\[
\Big(\frac{1-x(\theta)}{2}\Big)^{\alpha}=\frac{1-t(\theta)}{2}\quad\text{and}\quad\frac{1+x(\theta)}{2}=1-\Big(\frac{1-t(\theta)}{2}\Big)^{\frac{1}{\alpha}}.
\]
Hence the expression \eqref{Rod-int-biop-tha-Re} can be written as
\begin{align}\label{Rod-int-biop-tha-Re-ttha}
		P_n^{(\alpha,a,b)}\big(x(\theta)\big)=\mathrm{Re}\Bigg\{\frac{1}{\pi i}\int_{C_+}\Big\{\frac{(\xi-1)[1-(\frac{1-\xi}{2})^{\frac{1}{\alpha}}]}{\xi-t(\theta)}\Big\}^n\Big[\frac{1-\xi}{1-t(\theta)}\Big]^{\frac{a+1}{\alpha}-1}\Big[\frac{1-(\frac{1-\xi}{2})^{\frac{1}{\alpha}}}{1-(\frac{1-t(\theta)}{2})^{\frac{1}{\alpha}}}\Big]^b\frac{\mathrm{d}\xi}{\xi-t(\theta)}\Bigg\}.
\end{align}

For any $\theta\in(0,\pi)$ and $\varphi\in(0,\pi)$, define
\begin{align}\label{def-f}
		f_{\theta}(\varphi):=\log\frac{(\xi(\varphi)-1)[1-(\frac{1-\xi(\varphi)}{2})^{\frac{1}{\alpha}}]}{\xi(\varphi)-t(\theta)}=\log\frac{e^{i(\pi+\varphi)}\cdot[\Theta_{\frac{1}{\alpha}}(\varphi)]^{\alpha}\cdot[e^{i\frac{\pi-\varphi}{1+\alpha}}-\Theta_{\frac{1}{\alpha}}(\varphi)]}{[\Theta_{\frac{1}{\alpha}}(\varphi)]^{\alpha}\exp(-i\frac{\pi-\varphi}{1+\frac{1}{\alpha}})-[\Theta_{\frac{1}{\alpha}}(\theta)]^{\alpha}\,\Theta_{\alpha}(\theta)}
\end{align}
and
\begin{align}\label{def-g}
	\begin{split}
		g_{\theta}(\varphi)&:=\Big[\frac{1-\xi(\varphi)}{1-t(\theta)}\Big]^{\frac{a+1}{\alpha}-1}\Big[\frac{1-(\frac{1-\xi(\varphi)}{2})^{\frac{1}{\alpha}}}{1-(\frac{1-t(\theta)}{2})^{\frac{1}{\alpha}}}\Big]^b\frac{\xi'(\varphi)}{\xi(\varphi)-t(\theta)}\\
		&\,\,=\Bigg[\frac{\Theta_{\frac{1}{\alpha}}(\varphi)}{\Theta_{\frac{1}{\alpha}}(\theta)}\Bigg]^{a+1-\alpha}\times\frac{e^{-i[\pi+\frac{\pi-\varphi}{1+\alpha}(a+b+1-\alpha)]}}{2[\Theta_{\alpha}(\theta)]^{\frac{a+1}{\alpha}-1}}\times\frac{[e^{i\frac{\pi-\varphi}{1+\alpha}}-\Theta_{\frac{1}{\alpha}}(\varphi)]^b}{\{1-\Theta_{\frac{1}{\alpha}}(\theta)[\Theta_{\alpha}(\theta)]^{\frac{1}{\alpha}}\}^b}\\
		&\,\,\quad\times\frac{\xi'(\varphi)}{[\Theta_{\frac{1}{\alpha}}(\varphi)]^{\alpha}\exp(-i\frac{\pi-\varphi}{1+\frac{1}{\alpha}})-[\Theta_{\frac{1}{\alpha}}(\theta)]^{\alpha}\,\Theta_{\alpha}(\theta)}.
	\end{split}
\end{align}
Here the second equality in \eqref{def-f} and that in \eqref{def-g} are obtained by simple calculations using \eqref{parame-xiphi-up} and \eqref{def-ttha}. Since the curve $C_+$ is determined by the parametrization $\xi(\varphi)$ for $\varphi\in[0,\pi)$, we can write the expression \eqref{Rod-int-biop-tha-Re-ttha} as follows:
\begin{align}\label{Exp-Stedes}
	P_n^{(\alpha,a,b)}\big(x(\theta)\big)=\mathrm{Re}\Big\{\frac{1}{\pi i}\int_{0}^{\pi}e^{nf_{\theta}(\varphi)}g_{\theta}(\varphi)\mathrm{d}\varphi\Big\}.
\end{align}

\subsection{Conditions of steepest descent}
Based on the expression \eqref{Exp-Stedes}, we now discuss the conditions required for the method of steepest descent.

The first step in the method of steepest descent is to find the saddle point of the function $f_{\theta}(\varphi)$ for the asymptotic expansion, which is given in the following Lemma~\ref{lem-saddlepoint}.

\begin{lemma}\label{lem-saddlepoint}
	For any fixed $\alpha>0$ and $\theta\in(0,\pi)$, the equation $f_{\theta}'(\varphi)=0$ has a unique solution $\varphi=\theta$ in $(0,\pi)$.
\end{lemma}

The second key of the method of steepest descent is to investigate the monotonicity of the modulus $|e^{f_{\theta}(\varphi)}|$ near the saddle point $\varphi=\theta$. For this purpose, we define the function 
\begin{align}\label{def-T}
	T_{\theta}(\varphi):=\big|e^{f_{\theta}(\varphi)}\big|^2=\frac{[\Theta_{\frac{1}{\alpha}}(\varphi)]^{2\alpha}\cdot\{1+[\Theta_{\frac{1}{\alpha}}(\varphi)]^2-2\Theta_{\frac{1}{\alpha}}(\varphi)\cos\frac{\pi-\varphi}{1+\alpha}\}}{[\Theta_{\frac{1}{\alpha}}(\varphi)]^{2\alpha}-2[\Theta_{\frac{1}{\alpha}}(\theta)]^{\alpha}\,\Theta_{\alpha}(\theta)[\Theta_{\frac{1}{\alpha}}(\varphi)]^{\alpha}\cos\frac{\pi-\varphi}{1+\frac{1}{\alpha}}+[\Theta_{\frac{1}{\alpha}}(\theta)]^{2\alpha}\,[\Theta_{\alpha}(\theta)]^2}.
\end{align}
Here the second equality in \eqref{def-T} follows from \eqref{def-f}. The following Lemma~\ref{lem-mono-T} describes the monotonicity of $T_{\theta}(\varphi)$ in $(0,\pi)$.

\begin{lemma}\label{lem-mono-T}
	For any fixed $\alpha\geq1$ and $\theta\in(0,\pi)$, the function $T_{\theta}(\varphi)$ is strictly increasing on $(0,\theta)$ and strictly decreasing on $(\theta,\pi)$.
\end{lemma}

\begin{remark}\label{rem-2}
	Lemma~\ref{lem-mono-T} requires $\alpha\geq1$ because its proof will use Claim~\ref{clm-anaT'}, which holds only for $\alpha\geq1$; moreover, Remark~\ref{rem-3} following Claim~\ref{clm-anaT'} shows that it is false for $0<\alpha<1$. Consequently, Theorem~\ref{thm-Darbo-Formu-biop} also requires $\alpha\geq1$.
\end{remark}

We defer the proofs of Lemmas~\ref{lem-saddlepoint} and \ref{lem-mono-T} to Sections~\ref{prf-saddlepoint-lem} and \ref{prf-monoT-lem}, respectively. The following Corollary~\ref{Cor-secderiv} follows from Lemmas~\ref{lem-saddlepoint} and \ref{lem-mono-T}, which shows that $\mathrm{Re}f_{\theta}''(\theta)<0$. This is a necessary step in the computation during the method of steepest descent.

\begin{corollary}\label{Cor-secderiv}
	For any fixed $\alpha\geq1$ and $\theta\in(0,\pi)$, we have
	\[
	\mathrm{Re}f_{\theta}''(\theta)<0.
	\]
\end{corollary}

\begin{proof}
	It follows from the definition of $T_{\theta}(\varphi)$ in \eqref{def-T} that
	\begin{align}\label{exp-TT'}
		T_{\theta}(\varphi)=e^{2\mathrm{Re}f_{\theta}(\varphi)}\quad\text{and}\quad T_{\theta}'(\varphi)=2e^{2\mathrm{Re}f_{\theta}(\varphi)}\mathrm{Re}f_{\theta}'(\varphi).
	\end{align}
	Then, one can show that
	\[
	T_{\theta}''(\varphi)=2e^{2\mathrm{Re}f_{\theta}(\varphi)}\big\{2[\mathrm{Re}f_{\theta}'(\varphi)]^2+\mathrm{Re}f_{\theta}''(\varphi)\big\}.
	\]
	Noticing that $\mathrm{Re}f_{\theta}'(\theta)=0$ by Lemma~\ref{lem-saddlepoint}, and setting $\varphi=\theta$, we have
	\[
	T_{\theta}''(\theta)=2e^{2\mathrm{Re}f_{\theta}(\theta)}\mathrm{Re}f_{\theta}''(\theta).
	\]
	From Lemma~\ref{lem-mono-T}, we know that $\varphi=\theta$ is a maximum point of $T_{\theta}(\varphi)$ and hence $T_{\theta}''(\theta)<0$, which gives $\mathrm{Re}f_{\theta}''(\theta)<0$.
\end{proof}

\subsection{Partition of the curve $C_+$}\label{Sec-del}
Fix $\delta\in(0,1/6)$, and why $\delta$ cannot exceed $1/6$ will become clear in the proofs of Lemma~\ref{est-int-cen} and \ref{est-int-r+l}. Then, for sufficiently large $n$, we divide the interval $(0,\pi)$ into three parts: a neighborhood of the saddle point $\varphi=\theta$, and the two parts away from the saddle point (one on each side). Specifically,
\begin{align}\label{div-int}
	(0,\pi)=(0,\theta-n^{-\frac{1}{2}+\delta}]\cup(\theta-n^{-\frac{1}{2}+\delta},\theta+n^{-\frac{1}{2}+\delta})\cup[\theta+n^{-\frac{1}{2}+\delta},\pi).
\end{align}
The following Figure~\ref{div-cur} shows the corresponding partition of the curve $C_+$ under the partition \eqref{div-int} (here we take $\alpha=2$, $\theta=\pi/3$, $\delta=1/12$ and $n=100$ as an illustration).

\begin{figure}[htbp]
	\centering
	\includegraphics[width=0.6\linewidth]{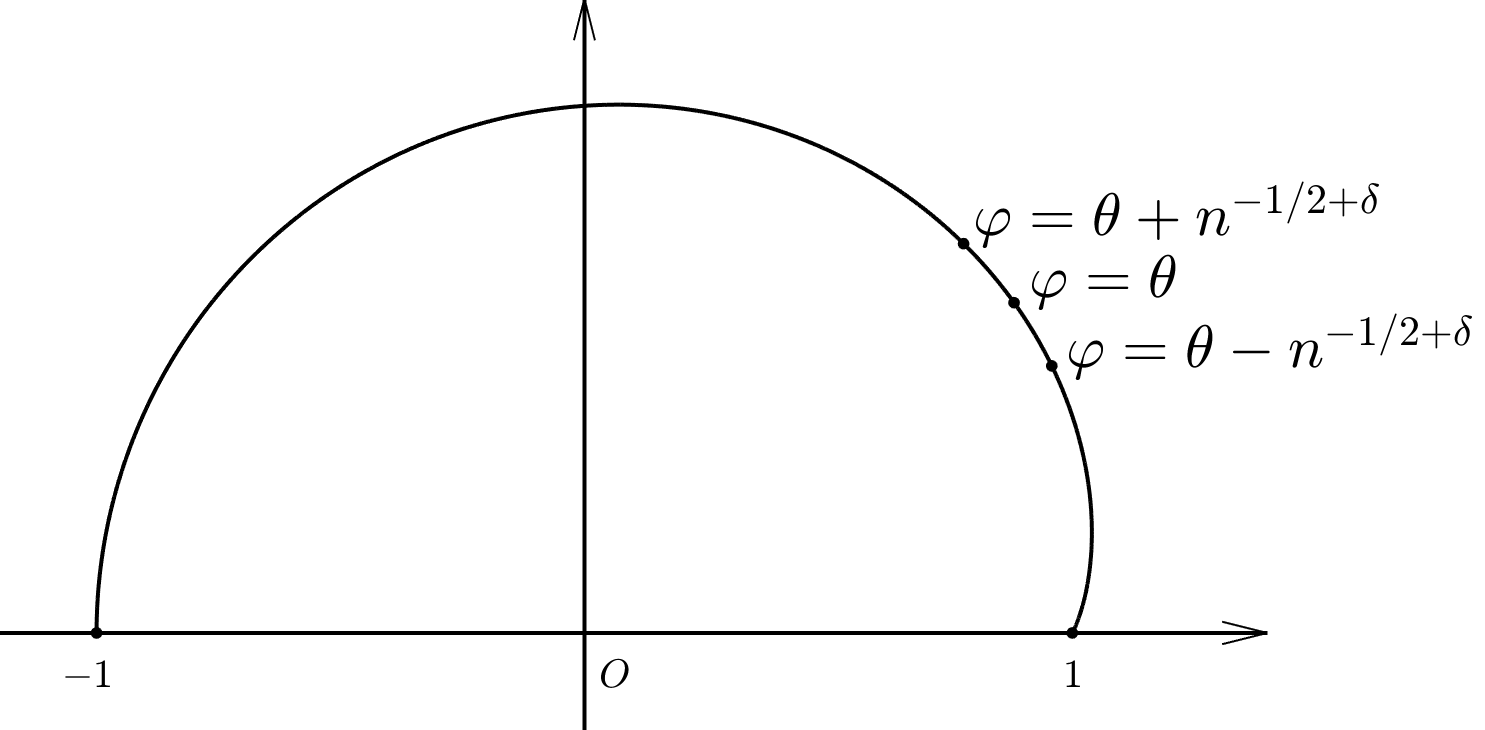}
	\caption{The partition of the curve $C_+$ induced by \eqref{div-int}.}
	\label{div-cur}
\end{figure}

We now consider the following three integrals:
\begin{align}\label{def-int-cen}
\mathcal{I}_{\mathrm{center}}(\theta,n):=\frac{1}{\pi i}\int_{\theta-n^{-\frac{1}{2}+\delta}}^{\theta+n^{-\frac{1}{2}+\delta}}e^{nf_{\theta}(\varphi)}g_{\theta}(\varphi)\mathrm{d}\varphi,
\end{align}
\begin{align}\label{def-int-r+l}
\mathcal{I}_{\mathrm{left}}(\theta,n):=\frac{1}{\pi i}\int_{0}^{\theta-n^{-\frac{1}{2}+\delta}}e^{nf_{\theta}(\varphi)}g_{\theta}(\varphi)\mathrm{d}\varphi,\quad\mathcal{I}_{\mathrm{right}}(\theta,n):=\frac{1}{\pi i}\int_{\theta+n^{-\frac{1}{2}+\delta}}^{\pi}e^{nf_{\theta}(\varphi)}g_{\theta}(\varphi)\mathrm{d}\varphi.
\end{align}
Therefore, by \eqref{Exp-Stedes}, we get
\begin{align}\label{Exp-Stedes-div}
	P_n^{(\alpha,a,b)}\big(x(\theta)\big)=\mathrm{Re}\Big\{\mathcal{I}_{\mathrm{center}}(\theta,n)+\mathcal{I}_{\mathrm{left}}(\theta,n)+\mathcal{I}_{\mathrm{right}}(\theta,n)\Big\}.
\end{align}
We will estimate these three integrals according to the conditions for steepest descent given above, namely Lemma~\ref{lem-saddlepoint}, Lemma~\ref{lem-mono-T} and Corollary~\ref{Cor-secderiv}.

\subsubsection{Estimate of $\mathcal{I}_{\mathrm{center}}(\theta,n)$}
The result for the first integral \eqref{def-int-cen} is as follows:

\begin{lemma}\label{est-int-cen}
	For any fixed $\alpha\geq1$, $a,b>-1$ and $\delta\in(0,1/6)$, as $n\to\infty$, we have
	\[
	\mathcal{I}_{\mathrm{center}}(\theta,n)=\sqrt{2}\pi^{-\frac{1}{2}}n^{-\frac{1}{2}}\frac{g_{\theta}(\theta)}{\sqrt{f_{\theta}''(\theta)}}e^{nf_{\theta}(\theta)}\big[1+O(n^{-1})\big],
	\]
	where the bound of the error term holds uniformly for $\theta\in[\varepsilon,\pi-\varepsilon]$ with any fixed $0<\varepsilon<\frac{\pi}{2}$. 
\end{lemma}

Although the detailed proof of Lemma~\ref{est-int-cen} will be deferred to Section~\ref{prf-est-int-cen-lem}, we first give a brief outline of its proof here. By Lemma~\ref{lem-saddlepoint}, we have $f_{\theta}'(\theta)=0$. Hence, expanding the integrand in the integral \eqref{def-int-cen} yields
\begin{align*}
	\mathcal{I}_{\mathrm{center}}(\theta,n)&=\frac{1}{\pi i}\int_{\theta-n^{-\frac{1}{2}+\delta}}^{\theta+n^{-\frac{1}{2}+\delta}}e^{n[f_{\theta}(\varphi)+\frac{f_{\theta}''(\theta)}{2}(\varphi-\theta)^2+\mathrm{ErrorTerm}]}\times\big[g_{\theta}(\theta)+\mathrm{ErrorTerm}\big]\mathrm{d}\varphi\\
	&=\frac{1}{\pi i}g_{\theta}(\theta)e^{nf_{\theta}(\theta)}\int_{\theta-n^{-\frac{1}{2}+\delta}}^{\theta+n^{-\frac{1}{2}+\delta}}e^{\frac{f_{\theta}''(\theta)}{2}n(\varphi-\theta)^2}\times e^{\mathrm{ErrorTerm}}\times\big[1+\mathrm{ErrorTerm}\big]\mathrm{d}\varphi.
\end{align*}
Make the change of variable $\varphi=\theta+n^{-\frac{1}{2}}\rho$, then
\begin{align*}
	\mathcal{I}_{\mathrm{center}}(\theta,n)&=\frac{1}{\pi i}n^{-\frac{1}{2}}g_{\theta}(\theta)e^{nf_{\theta}(\theta)}\int_{-n^{\delta}}^{n^{\delta}}e^{\frac{f_{\theta}''(\theta)}{2}\rho^2}\times e^{\mathrm{ErrorTerm}}\times\big[1+\mathrm{ErrorTerm}\big]\mathrm{d}\rho\\
	&=\frac{1}{\pi i}n^{-\frac{1}{2}}g_{\theta}(\theta)e^{nf_{\theta}(\theta)}\int_{-n^{\delta}}^{n^{\delta}}e^{\frac{f_{\theta}''(\theta)}{2}\rho^2}\times\big[1+\mathrm{ErrorTerm}\big]\times\big[1+\mathrm{ErrorTerm}\big]\mathrm{d}\rho.
\end{align*}
Hence based on $\mathrm{Re}f_{\theta}''(\theta)<0$ from Corollary~\ref{Cor-secderiv}, we can show that
\begin{align*}
	\mathcal{I}_{\mathrm{center}}(\theta,n)&=\frac{1}{\pi i}n^{-\frac{1}{2}}g_{\theta}(\theta)e^{nf_{\theta}(\theta)}\int_{-\infty}^{+\infty}e^{\frac{f_{\theta}''(\theta)}{2}\rho^2}\mathrm{d}\rho\times\big[1+\mathrm{ErrorTerm}\big]\\
	&=\sqrt{2}\pi^{-\frac{1}{2}}n^{-\frac{1}{2}}\frac{g_{\theta}(\theta)}{\sqrt{f_{\theta}''(\theta)}}e^{nf_{\theta}(\theta)}\times\big[1+\mathrm{ErrorTerm}\big].
\end{align*}
The error terms above depend on the corresponding expansions, and their detailed estimation in Section~\ref{prf-est-int-cen-lem} will reveal the role of $\delta\in(0,1/6)$.

\subsubsection{Estimate of $\mathcal{I}_{\mathrm{left}}(\theta,n)+\mathcal{I}_{\mathrm{right}}(\theta,n)$}
For the second and third integrals \eqref{def-int-r+l}, the result is as follows:

\begin{lemma}\label{est-int-r+l}
	For any fixed $\alpha\geq1$, $a,b>-1$ and $\delta\in(0,1/6)$, as $n\to\infty$, we have
	\[
	\mathcal{I}_{\mathrm{left}}(\theta,n)+\mathcal{I}_{\mathrm{right}}(\theta,n)=\sqrt{2}\pi^{-\frac{1}{2}}n^{-\frac{1}{2}}\frac{g_{\theta}(\theta)}{\sqrt{f_{\theta}''(\theta)}}e^{nf_{\theta}(\theta)}\cdot o(n^{-1}),
	\]
	where the error term tends to zero uniformly for $\theta\in[\varepsilon,\pi-\varepsilon]$ with any fixed $0<\varepsilon<\frac{\pi}{2}$. 
\end{lemma}

Similarly, the detailed proof of Lemma~\ref{est-int-r+l} will be given in Section~\ref{prf-est-int-r+l-lem}, and here we briefly outline its proof. We only illustrate the integral $\mathcal{I}_{\mathrm{left}}(\theta,n)$. Note the definition of $T_{\theta}(\varphi)$ in \eqref{def-T}. It follows from \eqref{def-int-r+l} that
\[
|\mathcal{I}_{\mathrm{left}}(\theta,n)|\leq\frac{1}{\pi}\int_{0}^{\theta-n^{-\frac{1}{2}+\delta}}\big|e^{nf_{\theta}(\varphi)}\big|\cdot|g_{\theta}(\varphi)|\mathrm{d}\varphi=\frac{1}{\pi}\int_{0}^{\theta-n^{-\frac{1}{2}+\delta}}\big[T_{\theta}(\varphi)\big]^{\frac{n}{2}}\cdot|g_{\theta}(\varphi)|\mathrm{d}\varphi.
\]
Since the function $T_{\theta}(\varphi)$ is strictly increasing on $(0,\theta-n^{-\frac{1}{2}+\delta})$ by Lemma~\ref{lem-mono-T}, we get
\[
|\mathcal{I}_{\mathrm{left}}(\theta,n)|\leq\big[T_{\theta}(\theta-n^{-\frac{1}{2}+\delta})\big]^{\frac{n}{2}}\frac{1}{\pi}\int_{0}^{\theta-n^{-\frac{1}{2}+\delta}}|g_{\theta}(\varphi)|\mathrm{d}\varphi=\big[T_{\theta}(\theta-n^{-\frac{1}{2}+\delta})\big]^{\frac{n}{2}}\times\mathrm{ErrorTerm}.
\]
Furthermore, one can show that
\[
\big[T_{\theta}(\theta-n^{-\frac{1}{2}+\delta})\big]^{\frac{n}{2}}=e^{n\mathrm{Re}f_{\theta}(\theta-n^{-\frac{1}{2}+\delta})}=n^{-\frac{1}{2}}e^{n\mathrm{Re}f_{\theta}(\theta)}\times\mathrm{ErrorTerm}.
\]
Thus we conclude that
\[
\mathcal{I}_{\mathrm{left}}(\theta,n)=n^{-\frac{1}{2}}e^{nf_{\theta}(\theta)}\times\mathrm{ErrorTerm}=\sqrt{2}\pi^{-\frac{1}{2}}n^{-\frac{1}{2}}\frac{g_{\theta}(\theta)}{\sqrt{f_{\theta}''(\theta)}}e^{nf_{\theta}(\theta)}\times\mathrm{ErrorTerm}.
\]
Detailed estimates for $T_{\theta}(\theta-n^{-\frac{1}{2}+\delta})$ and the above error terms will be provided in Section~\ref{prf-est-int-r+l-lem}, and their detailed estimation will also reveal the role of $\delta\in(0,1/6)$.

\subsection{Conclusion: proof of Theorem~\ref{thm-Darbo-Formu-biop}}
Recall the definition of $x(\theta)$ from \eqref{def-xtha}:
\[
x(\theta)=1-2\Theta_{\frac{1}{\alpha}}(\theta)\big[\Theta_{\alpha}(\theta)\big]^{\frac{1}{\alpha}},
\]
where
\[
\Theta_{\alpha}(\theta)=\frac{\sin\theta}{(1+\alpha)\sin\frac{\pi-\theta}{1+\alpha}}\quad\text{and}\quad\Theta_{\frac{1}{\alpha}}(\theta)=\frac{\sin\theta}{(1+\frac{1}{\alpha})\sin\frac{\pi-\theta}{1+\frac{1}{\alpha}}}.
\]
Also recall that the functions $f_{\theta}(\varphi)$ and $g_{\theta}(\theta)$ in \eqref{def-f} and \eqref{def-g} involve $\xi(\varphi)$ and $\Theta_{\frac{1}{\alpha}}(\varphi)$, defined in \eqref{parame-xiphi-up} and \eqref{def-Tha-phi} respectively.

\begin{proposition}\label{est-int-cen+r+l}
	For any fixed $\alpha\geq1$ and $a,b>-1$, as $n\to\infty$, we have
	\[
	P_n^{(\alpha, a,b)}\big(x(\theta)\big)=\sqrt{2}\pi^{-\frac{1}{2}}n^{-\frac{1}{2}}\Big(\mathrm{Re}\Big\{\frac{g_{\theta}(\theta)}{\sqrt{f_{\theta}''(\theta)}}e^{nf_{\theta}(\theta)}\Big\}\Big)\big[1+O(n^{-1})\big],
	\]
	where the bound of the error term holds uniformly for $\theta\in[\varepsilon,\pi-\varepsilon]$ with any fixed $0<\varepsilon<\frac{\pi}{2}$. 
\end{proposition}

\begin{proof}
Proposition~\ref{est-int-cen+r+l} follows directly from expression \eqref{Exp-Stedes-div} together with Lemma~\ref{est-int-cen} and Lemma~\ref{est-int-r+l}.
\end{proof}

We next compute $f_{\theta}(\theta)$, $g_{\theta}(\theta)$ and $f_{\theta}''(\theta)$, and then use them in Proposition~\ref{est-int-cen+r+l} to obtain Theorem~\ref{thm-Darbo-Formu-biop}.

\subsubsection{Calculation of $f_{\theta}(\theta)$}
Taking $\varphi=\theta$ in the expression for $f_{\theta}(\varphi)$ in \eqref{def-f}, we have
\[
f_{\theta}(\theta)=\log\frac{e^{i(\pi+\theta)}\cdot[\Theta_{\frac{1}{\alpha}}(\theta)]^{\alpha}\cdot[e^{i\frac{\pi-\theta}{1+\alpha}}-\Theta_{\frac{1}{\alpha}}(\theta)]}{[\Theta_{\frac{1}{\alpha}}(\theta)]^{\alpha}\exp(-i\frac{\pi-\theta}{1+\frac{1}{\alpha}})-[\Theta_{\frac{1}{\alpha}}(\theta)]^{\alpha}\,\Theta_{\alpha}(\theta)}=\log\frac{e^{i(\pi+\theta)}\cdot[e^{i\frac{\pi-\theta}{1+\alpha}}-\Theta_{\frac{1}{\alpha}}(\theta)]}{e^{-i\frac{\pi-\theta}{1+\frac{1}{\alpha}}}-\Theta_{\alpha}(\theta)},
\]
that is,
\[
f_{\theta}(\theta)=\log\Bigg|\frac{e^{i\frac{\pi-\theta}{1+\alpha}}-\Theta_{\frac{1}{\alpha}}(\theta)}{e^{-i\frac{\pi-\theta}{1+\frac{1}{\alpha}}}-\Theta_{\alpha}(\theta)}\Bigg|+i\arg\Bigg\{e^{i(\pi+\theta)}\frac{e^{i\frac{\pi-\theta}{1+\alpha}}-\Theta_{\frac{1}{\alpha}}(\theta)}{e^{-i\frac{\pi-\theta}{1+\frac{1}{\alpha}}}-\Theta_{\alpha}(\theta)}\Bigg\}.
\]
Note that
\begin{align}\label{cal-Aalp}
	\frac{e^{i\frac{\pi-\theta}{1+\alpha}}-\Theta_{\frac{1}{\alpha}}(\theta)}{e^{-i\frac{\pi-\theta}{1+\frac{1}{\alpha}}}-\Theta_{\alpha}(\theta)}=-\frac{\sin\frac{\pi-\theta}{1+\alpha}}{\sin\frac{\pi-\theta}{1+\frac{1}{\alpha}}}.
\end{align}
Here we leave the calculation of identity \eqref{cal-Aalp} to Appendix~\ref{prf-Aalp}. Hence we get
\begin{align}\label{expre-f-MA}
	f_{\theta}(\theta)=\log\frac{\sin\frac{\pi-\theta}{1+\alpha}}{\sin\frac{\pi-\theta}{1+\frac{1}{\alpha}}}+i\theta.
\end{align}

\subsubsection{Calculation of $g_{\theta}(\theta)$}
Similarly, setting $\varphi=\theta$ in \eqref{def-g}, we have
\begin{align*}
		g_{\theta}(\theta)&=\Bigg[\frac{\Theta_{\frac{1}{\alpha}}(\theta)}{\Theta_{\frac{1}{\alpha}}(\theta)}\Bigg]^{a+1-\alpha}\times\frac{e^{-i[\pi+\frac{\pi-\theta}{1+\alpha}(a+b+1-\alpha)]}}{2[\Theta_{\alpha}(\theta)]^{\frac{a+1}{\alpha}-1}}\times\frac{[e^{i\frac{\pi-\theta}{1+\alpha}}-\Theta_{\frac{1}{\alpha}}(\theta)]^b}{\{1-\Theta_{\frac{1}{\alpha}}(\theta)[\Theta_{\alpha}(\theta)]^{\frac{1}{\alpha}}\}^b}\\
		&\quad\times\frac{\xi'(\theta)}{[\Theta_{\frac{1}{\alpha}}(\theta)]^{\alpha}\exp(-i\frac{\pi-\theta}{1+\frac{1}{\alpha}})-[\Theta_{\frac{1}{\alpha}}(\theta)]^{\alpha}\,\Theta_{\alpha}(\theta)}.
\end{align*}
After simplification and calculation, we obtain
\begin{align}\label{expre-gthatha}
	g_{\theta}(\theta)=\frac{e^{-i[\pi+\frac{\pi-\theta}{1+\alpha}(a+b+1-\alpha)]}\cdot[e^{i\frac{\pi-\theta}{1+\alpha}}-\Theta_{\frac{1}{\alpha}}(\theta)]^b\cdot[e^{i\frac{\pi-\theta}{1+\frac{1}{\alpha}}}-\Theta_{\alpha}(\theta)]\cdot\xi'(\theta)}{2[\Theta_{\frac{1}{\alpha}}(\theta)]^{\alpha}\,[\Theta_{\alpha}(\theta)]^{\frac{a+1}{\alpha}-1}\cdot\{1-\Theta_{\frac{1}{\alpha}}(\theta)[\Theta_{\alpha}(\theta)]^{\frac{1}{\alpha}}\}^b\cdot\{1+[\Theta_{\alpha}(\theta)]^2-2\Theta_{\alpha}(\theta)\cos\frac{\pi-\theta}{1+\frac{1}{\alpha}}\}}.
\end{align}

\subsubsection{Calculation of $f_{\theta}''(\theta)$}
Let us define
\begin{align}\label{def-bigF}
	F_{\theta}(\xi):=\log\frac{(\xi-1)[1-(\frac{1-\xi}{2})^{\frac{1}{\alpha}}]}{\xi-t(\theta)},
\end{align}
then, it follows from the definition of $f_{\theta}(\varphi)$ in \eqref{def-f} that
\begin{align}\label{redef-f}
	f_{\theta}(\varphi)=F_{\theta}\big(\xi(\varphi)\big).
\end{align}
From \eqref{def-bigF}, simple calculation and simplification give
\begin{align*}
	F_{\theta}'(\xi)&=\frac{1}{\xi-1}+\frac{(\frac{1-\xi}{2})^{\frac{1}{\alpha}}}{\alpha(1-\xi)[1-(\frac{1-\xi}{2})^{\frac{1}{\alpha}}]}-\frac{1}{\xi-t(\theta)}\\
	&=\frac{\alpha[(1+\frac{1}{\alpha})(\frac{1-\xi}{2})^{\frac{1}{\alpha}}-1][1-t(\theta)]-2(\frac{1-\xi}{2})^{1+\frac{1}{\alpha}}}{\alpha(1-\xi)[1-(\frac{1-\xi}{2})^{\frac{1}{\alpha}}][\xi-t(\theta)]}.
\end{align*}
Let
\begin{align}\label{def-upp}
	\mathrm{Upp}_{\theta}(\xi):=\alpha[(1+\frac{1}{\alpha})(\frac{1-\xi}{2})^{\frac{1}{\alpha}}-1][1-t(\theta)]-2(\frac{1-\xi}{2})^{1+\frac{1}{\alpha}}
\end{align}
and
\begin{align}\label{def-low}
	\mathrm{Low}_{\theta}(\xi):=\alpha(1-\xi)[1-(\frac{1-\xi}{2})^{\frac{1}{\alpha}}][\xi-t(\theta)].
\end{align}
Then, we can write
\[
F_{\theta}'(\xi)=\frac{\mathrm{Upp}_{\theta}(\xi)}{\mathrm{Low}_{\theta}(\xi)}.
\]
Thus it follows from \eqref{redef-f} that
\begin{align}\label{expre-fxiphi'}
	f_{\theta}'(\varphi)=F_{\theta}'\big(\xi(\varphi)\big)\xi'(\varphi)=\frac{\mathrm{Upp}_{\theta}\big(\xi(\varphi)\big)}{\mathrm{Low}_{\theta}\big(\xi(\varphi)\big)}\xi'(\varphi).
\end{align}

Set $\varphi=\theta$ in \eqref{expre-fxiphi'}, then
\[
f_{\theta}'(\theta)=\frac{\mathrm{Upp}_{\theta}\big(\xi(\theta)\big)}{\mathrm{Low}_{\theta}\big(\xi(\theta)\big)}\xi'(\theta).
\]
Since $\xi'(\theta)\neq0$ by Lemma~\ref{cla-xi'} and $f_{\theta}'(\theta)=0$ by Lemma~\ref{lem-saddlepoint}, the formula \eqref{expre-fxiphi'} yields
\[
\mathrm{Upp}_{\theta}\big(\xi(\theta)\big)=0.
\]
Based on \eqref{expre-fxiphi'}, we have
\[
f_{\theta}''(\varphi)=\frac{\mathrm{Upp}_{\theta}'\big(\xi(\varphi)\big)\mathrm{Low}_{\theta}\big(\xi(\varphi)\big)-\mathrm{Upp}_{\theta}\big(\xi(\varphi)\big)\mathrm{Low}_{\theta}'\big(\xi(\varphi)\big)}{\mathrm{Low}_{\theta}^2\big(\xi(\varphi)\big)}\big[\xi'(\varphi)\big]^2+\frac{\mathrm{Upp}_{\theta}\big(\xi(\varphi)\big)}{\mathrm{Low}_{\theta}\big(\xi(\varphi)\big)}\xi''(\varphi).
\]
Now that $\mathrm{Upp}_{\theta}\big(\xi(\theta)\big)=0$, we obtain
\begin{align}\label{expre-fxiphi''two}
	f_{\theta}''(\theta)=\frac{\mathrm{Upp}_{\theta}'\big(\xi(\theta)\big)}{\mathrm{Low}_{\theta}\big(\xi(\theta)\big)}\big[\xi'(\theta)\big]^2.
\end{align}

From \eqref{def-upp}, by a simple calculation and simplification, we have
\[
\mathrm{Upp}_{\theta}'(\xi)=\frac{1+\alpha}{2\alpha}(\frac{1-\xi}{2})^{\frac{1}{\alpha}-1}[t(\theta)-\xi].
\]
Combining this with \eqref{def-low} and simplifying, we get
\[
\frac{\mathrm{Upp}_{\theta}'(\xi)}{\mathrm{Low}_{\theta}(\xi)}=-\frac{(1+\alpha)(\frac{1-\xi}{2})^{\frac{1}{\alpha}-2}}{4\alpha^2[1-(\frac{1-\xi}{2})^{\frac{1}{\alpha}}]}.
\]
Hence by \eqref{expre-fxiphi''two}, we obtain
\begin{align}\label{expre-fxitha''two}
	f_{\theta}''(\theta)=-\frac{(1+\alpha)(\frac{1-\xi(\theta)}{2})^{\frac{1}{\alpha}-2}}{4\alpha^2[1-(\frac{1-\xi(\theta)}{2})^{\frac{1}{\alpha}}]}\big[\xi'(\theta)\big]^2.
\end{align}

Recall the expression of $\xi(\varphi)$ in \eqref{parame-xiphi-up}, set $\varphi=\theta$, then
\[
\xi(\theta)=1-2\big[\Theta_{\frac{1}{\alpha}}(\theta)\big]^{\alpha}\exp\Big(-i\frac{\pi-\theta}{1+\frac{1}{\alpha}}\Big).
\]
This combines \eqref{expre-fxitha''two} and implies that
\begin{align}\label{expre-fxitha''three}
	\begin{split}
		f_{\theta}''(\theta)&=-\frac{(1+\alpha)[\Theta_{\frac{1}{\alpha}}(\theta)]^{1-2\alpha}\exp(-i\frac{\pi-\theta}{1+\alpha}(1-2\alpha))}{4\alpha^2[1-\Theta_{\frac{1}{\alpha}}(\theta)\exp(-i\frac{\pi-\theta}{1+\alpha})]}\big[\xi'(\theta)\big]^2\\
		&=\frac{1+\alpha}{4\alpha^2}\cdot\frac{e^{-i(\pi-\frac{\pi-\theta}{1+\alpha}2\alpha)}\cdot[\Theta_{\frac{1}{\alpha}}(\theta)]^{1-2\alpha}\cdot[\xi'(\theta)]^2}{e^{i\frac{\pi-\theta}{1+\alpha}}-\Theta_{\frac{1}{\alpha}}(\theta)}.
	\end{split}
\end{align}

\subsubsection{Proof of Theorem~\ref{thm-Darbo-Formu-biop}}
It follows from \eqref{expre-fxitha''three} that
\begin{align}\label{expre-sqrf''thatha}
	\frac{1}{\sqrt{f_{\theta}''(\theta)}}=\frac{2\alpha}{\sqrt{1+\alpha}}\cdot\frac{[e^{i\frac{\pi-\theta}{1+\alpha}}-\Theta_{\frac{1}{\alpha}}(\theta)]^{\frac{1}{2}}}{e^{-i(\frac{\pi}{2}-\frac{\pi-\theta}{1+\alpha}\alpha)}\cdot[\Theta_{\frac{1}{\alpha}}(\theta)]^{\frac{1}{2}-\alpha}\cdot\xi'(\theta)}.
\end{align}
Recalling the definition of $M_{\alpha}(\theta)$ in \eqref{def-Balp}:
\[
M_{\alpha}(\theta)=\frac{e^{-i[\frac{\pi}{2}+\frac{\pi-\theta}{1+\alpha}(a+b+1)]}\cdot[e^{i\frac{\pi-\theta}{1+\alpha}}-\Theta_{\frac{1}{\alpha}}(\theta)]^{b+\frac{1}{2}}\cdot[e^{i\frac{\pi-\theta}{1+\frac{1}{\alpha}}}-\Theta_{\alpha}(\theta)]}{[\Theta_{\frac{1}{\alpha}}(\theta)]^{\frac{1}{2}}\,[\Theta_{\alpha}(\theta)]^{\frac{a+1}{\alpha}-1}\cdot\{1-\Theta_{\frac{1}{\alpha}}(\theta)[\Theta_{\alpha}(\theta)]^{\frac{1}{\alpha}}\}^b\cdot\{1+[\Theta_{\alpha}(\theta)]^2-2\Theta_{\alpha}(\theta)\cos\frac{\pi-\theta}{1+\frac{1}{\alpha}}\}},
\]
multiplying \eqref{expre-gthatha} and \eqref{expre-sqrf''thatha}, and comparing the result with the expression for $M_{\alpha}(\theta)$, we obtain
\begin{align}\label{expre-gf''-B}
	\frac{g_{\theta}(\theta)}{\sqrt{f_{\theta}''(\theta)}}=\frac{\alpha}{\sqrt{1+\alpha}}M_{\alpha}(\theta).
\end{align}

Proposition~\ref{est-int-cen+r+l} together with \eqref{expre-f-MA} and \eqref{expre-gf''-B} implies that
\begin{align*}
	P_n^{(\alpha, a,b)}\big(x(\theta)\big)&=\sqrt{2}\pi^{-\frac{1}{2}}n^{-\frac{1}{2}}\Big(\mathrm{Re}\Big\{\frac{g_{\theta}(\theta)}{\sqrt{f_{\theta}''(\theta)}}e^{nf_{\theta}(\theta)}\Big\}\Big)\big[1+O(n^{-1})\big]\\
	&=\sqrt{2}\pi^{-\frac{1}{2}}n^{-\frac{1}{2}}\Big(\mathrm{Re}\Big\{\frac{\alpha}{\sqrt{1+\alpha}}M_{\alpha}(\theta)e^{n[\log(\sin\frac{\pi-\theta}{1+\alpha}/\sin\frac{\pi-\theta}{1+\frac{1}{\alpha}})+i\theta]}\Big\}\Big)\big[1+O(n^{-1})\big],
\end{align*}
that is,
\[
P_n^{(\alpha, a,b)}\big(x(\theta)\big)=\frac{\sqrt{2}\,\alpha}{\sqrt{1+\alpha}}\pi^{-\frac{1}{2}}n^{-\frac{1}{2}}\Bigg(\frac{\sin\frac{\pi-\theta}{1+\alpha}}{\sin\frac{\pi-\theta}{1+\frac{1}{\alpha}}}\Bigg)^n\Big(\mathrm{Re}\big\{M_{\alpha}(\theta)e^{in\theta}\big\}\Big)\big[1+O(n^{-1})\big].
\]
This completes the proof of Theorem~\ref{thm-Darbo-Formu-biop}.

\section{Verify conditions for steepest descent}
In this section, we are going to verify the conditions for the method of steepest descent. Namely, we will prove Lemmas~\ref{lem-saddlepoint} and \ref{lem-mono-T}.

\subsection{Proof of Lemma~\ref{lem-saddlepoint}}\label{prf-saddlepoint-lem}
By \eqref{expre-fxiphi'}, we have
\[
f_{\theta}'(\varphi)=\frac{\mathrm{Upp}_{\theta}\big(\xi(\varphi)\big)}{\mathrm{Low}_{\theta}\big(\xi(\varphi)\big)}\xi'(\varphi).
\]
Our aim is to prove that the equation $f_{\theta}'(\varphi)=0$ has a unique solution $\varphi=\theta$ in $(0,\pi)$. Since $\xi'(\theta)\neq0$ by Lemma~\ref{cla-xi'}, we turn to show that the equation $\mathrm{Upp}_{\theta}\big(\xi(\varphi)\big)=0$ has a unique solution $\varphi=\theta$ in $(0,\pi)$.

The expression of $\mathrm{Upp}_{\theta}\big(\xi(\varphi)\big)$ is given in \eqref{def-upp}:
\begin{align}\label{cal-1}
	\mathrm{Upp}_{\theta}\big(\xi(\varphi)\big)=\alpha\big[(1+\tfrac{1}{\alpha})\big(\tfrac{1-\xi(\varphi)}{2}\big)^{\frac{1}{\alpha}}-1\big]\big[1-t(\theta)\big]-2\big(\tfrac{1-\xi(\varphi)}{2}\big)^{1+\frac{1}{\alpha}}
\end{align}
with $\xi(\varphi)$ and $t(\theta)$ defined in \eqref{parame-xiphi-up} and \eqref{def-ttha} respectively by
\[
\xi(\varphi)=1-2\big[\Theta_{\frac{1}{\alpha}}(\varphi)\big]^{\alpha}\exp\Big(-i\frac{\pi-\varphi}{1+\frac{1}{\alpha}}\Big)\quad\text{and}\quad t(\theta)=1-2\big[\Theta_{\frac{1}{\alpha}}(\theta)\big]^{\alpha}\,\Theta_{\alpha}(\theta).
\]

Noticing that
\[
\frac{1-\xi(\varphi)}{2}=\big[\Theta_{\frac{1}{\alpha}}(\varphi)\big]^{\alpha}\exp\Big(-i\frac{\pi-\varphi}{1+\frac{1}{\alpha}}\Big),
\]
one can show that
\begin{align*}
	\alpha&\big[(1+\tfrac{1}{\alpha})\big(\tfrac{1-\xi(\varphi)}{2}\big)^{\frac{1}{\alpha}}-1\big]=\alpha\big[(1+\tfrac{1}{\alpha})\Theta_{\frac{1}{\alpha}}(\varphi)\exp\Big(-i\frac{\pi-\varphi}{1+\alpha}\Big)-1\big]\\
	&=\alpha\Big[(1+\tfrac{1}{\alpha})\Theta_{\frac{1}{\alpha}}(\varphi)\exp\Big(i\frac{\pi-\varphi}{1+\frac{1}{\alpha}}\Big)-\exp\big(i(\pi-\varphi)\big)\Big]\exp\big(-i(\pi-\varphi)\big).
\end{align*}
Now, by \eqref{def-Tha-phi},
\[
\Theta_{\frac{1}{\alpha}}(\varphi)=\frac{\sin\varphi}{(1+\frac{1}{\alpha})\sin\frac{\pi-\varphi}{1+\frac{1}{\alpha}}}=\frac{\sin(\pi-\varphi)}{(1+\frac{1}{\alpha})\sin\frac{\pi-\varphi}{1+\frac{1}{\alpha}}},
\]
thus we get
\begin{align*}
	\alpha\big[(1+\tfrac{1}{\alpha})\big(\tfrac{1-\xi(\varphi)}{2}\big)^{\frac{1}{\alpha}}-1\big]&=\alpha\Big[\frac{\sin(\pi-\varphi)}{\sin\frac{\pi-\varphi}{1+\frac{1}{\alpha}}}\cos\frac{\pi-\varphi}{1+\frac{1}{\alpha}}-\cos(\pi-\varphi)\Big]\exp\big(-i(\pi-\varphi)\big)\\
	&=\alpha\frac{\sin\frac{\pi-\varphi}{1+\alpha}}{\sin\frac{\pi-\varphi}{1+\frac{1}{\alpha}}}\exp\big(-i(\pi-\varphi)\big).
\end{align*}
This implies that
\begin{align}\label{cal-2}
	\alpha\big[(1+\tfrac{1}{\alpha})\big(\tfrac{1-\xi(\varphi)}{2}\big)^{\frac{1}{\alpha}}-1\big]=\frac{\Theta_{\frac{1}{\alpha}}(\varphi)}{\Theta_{\alpha}(\varphi)}\exp\big(-i(\pi-\varphi)\big).
\end{align}

In addition, we have
\begin{align}\label{cal-3}
	1-t(\theta)=2\big[\Theta_{\frac{1}{\alpha}}(\theta)\big]^{\alpha}\,\Theta_{\alpha}(\theta)\quad\text{and}\quad2\big(\tfrac{1-\xi(\varphi)}{2}\big)^{1+\frac{1}{\alpha}}=2\big[\Theta_{\frac{1}{\alpha}}(\varphi)\big]^{1+\alpha}\exp\big(-i(\pi-\varphi)\big).
\end{align}
Thus, combining \eqref{cal-1} with \eqref{cal-2} and \eqref{cal-3}, we obtain
\begin{align*}
	\mathrm{Upp}_{\theta}\big(\xi(\varphi)\big)&=2\frac{\Theta_{\frac{1}{\alpha}}(\varphi)}{\Theta_{\alpha}(\varphi)}\exp\big(-i(\pi-\varphi)\big)\big[\Theta_{\frac{1}{\alpha}}(\theta)\big]^{\alpha}\,\Theta_{\alpha}(\theta)-2\big[\Theta_{\frac{1}{\alpha}}(\varphi)\big]^{1+\alpha}\exp\big(-i(\pi-\varphi)\big)\\
	&=-2\frac{\Theta_{\frac{1}{\alpha}}(\varphi)}{\Theta_{\alpha}(\varphi)}\Big\{\big[\Theta_{\frac{1}{\alpha}}(\varphi)\big]^{\alpha}\,\Theta_{\alpha}(\varphi)-\big[\Theta_{\frac{1}{\alpha}}(\theta)\big]^{\alpha}\,\Theta_{\alpha}(\theta)\Big\}\exp\big(-i(\pi-\varphi)\big).
\end{align*}
Therefore, $\mathrm{Upp}_{\theta}\big(\xi(\varphi)\big)=0$ is equivalent to
\begin{align}\label{cal-4}
	\big[\Theta_{\frac{1}{\alpha}}(\varphi)\big]^{\alpha}\,\Theta_{\alpha}(\varphi)=\big[\Theta_{\frac{1}{\alpha}}(\theta)\big]^{\alpha}\,\Theta_{\alpha}(\theta).
\end{align}
We now show that the equation \eqref{cal-4} has a unique solution $\varphi=\theta$ in $(0,\pi)$. 

It follows from the proof of Lemma~\ref{lem-xtha} that the functions $\Theta_{\frac{1}{\alpha}}(\varphi)$ and $\Theta_{\alpha}(\varphi)$ are strictly increasing on $(0,\pi)$. This yields that the function $\big[\Theta_{\frac{1}{\alpha}}(\varphi)\big]^{\alpha}\,\Theta_{\alpha}(\varphi)$ is also strictly increasing on $(0,\pi)$. Hence the equation \eqref{cal-4} has a unique solution $\varphi=\theta$ in $(0,\pi)$.

This completes the proof of Lemma~\ref{lem-saddlepoint}.

\subsection{Proof of Lemma~\ref{lem-mono-T}}\label{prf-monoT-lem}
For convenience, let
\begin{align}\label{def-klrs}
	\begin{split}
		k(&\varphi):=\big[\Theta_{\frac{1}{\alpha}}(\varphi)\big]^{2\alpha},\quad l(\varphi):=1+\big[\Theta_{\frac{1}{\alpha}}(\varphi)\big]^2-2\Theta_{\frac{1}{\alpha}}(\varphi)\cos\frac{\pi-\varphi}{1+\alpha},\\
		&r(\varphi):=-2\big[\Theta_{\frac{1}{\alpha}}(\varphi)\big]^{\alpha}\cos\frac{\pi-\varphi}{1+\frac{1}{\alpha}},\quad s(\varphi):=\big[\Theta_{\frac{1}{\alpha}}(\varphi)\big]^{\alpha}\,\Theta_{\alpha}(\varphi).
	\end{split}
\end{align}
Then, by \eqref{def-T}, the function $T_{\theta}(\varphi)$ can be rewritten as
\begin{align}\label{rew-T}
	T_{\theta}(\varphi)=\frac{k(\varphi)l(\varphi)}{k(\varphi)+r(\varphi)s(\theta)+s^2(\theta)}.
\end{align}
We aim to prove that for any fixed $\theta\in(0,\pi)$, the function $T_{\theta}(\varphi)$ is strictly increasing on $(0,\theta)$ and strictly decreasing on $(\theta,\pi)$.

\subsubsection{Expression for $T_{\theta}'(\varphi)$}
The derivative of $T_{\theta}(\varphi)$ is
\begin{align*}
	T_{\theta}'(\varphi)&=\frac{(kl)'(\varphi)[k(\varphi)+r(\varphi)s(\theta)+s^2(\theta)]-(kl)(\varphi)[k'(\varphi)+r'(\varphi)s(\theta)]}{[k(\varphi)+r(\varphi)s(\theta)+s^2(\theta)]^2}\\
	&=\frac{(kl)'(\varphi)s^2(\theta)+[(kl)'r-klr'](\varphi)s(\theta)+(k^2l')(\varphi)}{[k(\varphi)+r(\varphi)s(\theta)+s^2(\theta)]^2}.
\end{align*}
Define
\begin{align}\label{def-uvw}
	u(\varphi):=(kl)'(\varphi),\quad v(\varphi):=[(kl)'r-klr'](\varphi)\quad\text{and}\quad w(\varphi):=(k^2l')(\varphi).
\end{align}
Then, we rewrite $T_{\theta}'(\varphi)$ as
\begin{align}\label{cal-T'}
	T_{\theta}'(\varphi)=\frac{u(\varphi)s^2(\theta)+v(\varphi)s(\theta)+w(\varphi)}{[k(\varphi)+r(\varphi)s(\theta)+s^2(\theta)]^2}.
\end{align}

Recall the second formula in \eqref{exp-TT'}:
\[
T_{\theta}'(\varphi)=2e^{2\mathrm{Re}f_{\theta}(\varphi)}\mathrm{Re}f_{\theta}'(\varphi).
\]
Noticing that $\mathrm{Re}f_{\theta}'(\theta)=0$ by Lemma~\ref{lem-saddlepoint}, and setting $\varphi=\theta$, we get
\[
T_{\theta}'(\theta)=0,
\]
which implies that
\[
u(\theta)s^2(\theta)+v(\theta)s(\theta)+w(\theta)=0.
\]
Since $\theta$ is arbitrary in $(0,\pi)$, we have
\[
u(\varphi)s^2(\varphi)+v(\varphi)s(\varphi)+w(\varphi)=0.
\]
It follows that
\begin{align}\label{cal-v}
	v(\varphi)=-u(\varphi)s(\varphi)-\frac{w(\varphi)}{s(\varphi)}.
\end{align}
Hence by \eqref{cal-T'} and \eqref{cal-v}, we have
\[
T_{\theta}'(\varphi)=\frac{u(\varphi)s^2(\theta)-[u(\varphi)s(\varphi)+\frac{w(\varphi)}{s(\varphi)}]s(\theta)+w(\varphi)}{[k(\varphi)+r(\varphi)s(\theta)+s^2(\theta)]^2}.
\]
After simplification, we obtain
\begin{align}\label{cal-sim-T'}
	T_{\theta}'(\varphi)=\frac{[s(\varphi)-s(\theta)][w(\varphi)-u(\varphi)s(\varphi)s(\theta)]}{s(\varphi)[k(\varphi)+r(\varphi)s(\theta)+s^2(\theta)]^2}.
\end{align}

\subsubsection{Analysis of $T_{\theta}'(\varphi)$}
It follows from the proof of Lemma~\ref{lem-xtha} that the functions $\Theta_{\frac{1}{\alpha}}(\varphi)$ and $\Theta_{\alpha}(\varphi)$ are strictly increasing on $(0,\pi)$ with $\Theta_{\frac{1}{\alpha}}(0^+)=\Theta_{\alpha}(0^+)=0$ and $\Theta_{\frac{1}{\alpha}}(\pi^-)=\Theta_{\alpha}(\pi^-)=1$. Hence by the definition of $s(\varphi)$ from the last formula in \eqref{def-klrs}, the function $s(\varphi)$ is strictly increasing on $(0,\pi)$ with $s(0^+)=0$ and $s(\pi^-)=1$. Thus we conclude that the equation
\[
s(\varphi)-s(\theta)=0
\]
has a unique solution $\varphi=\theta$ in $(0,\pi)$.

The following Claim~\ref{clm-anaT'} is key to analyzing the properties of $T_{\theta}'(\varphi)$, and its proof will be given in Section~\ref{prf-clm-anaT'}.

\begin{claim}\label{clm-anaT'}
	Fix $\alpha\geq1$. Then for any $\varphi\in(0,\pi)$, either
	\begin{itemize}
		\item $u(\varphi)=0$ and $w(\varphi)\neq0$, or
		\vspace{2mm}
		\item $u(\varphi)\neq0$ and $\frac{w(\varphi)}{u(\varphi)s(\varphi)}\notin(0,1)$.
	\end{itemize}
\end{claim}

\begin{remark}\label{rem-3}
	In the proof of Claim~\ref{clm-anaT'} for the second case $u(\varphi)\neq0$, the assumption $\alpha\geq1$ will be used in Lemma~\ref{lem-del} in Section~\ref{Analysisofhvarphi}. Moreover, Remark~\ref{rem-4} points out that Lemma~\ref{lem-del} fails when $0<\alpha<1$. Hence Lemma~\ref{lem-mono-T} requires $\alpha\geq1$, and consequently so does Theorem~\ref{thm-Darbo-Formu-biop}.
	
	More explicitly, Claim~\ref{clm-anaT'} fails for $0<\alpha<1$. In fact, we can show that when $0<\alpha<1$, there exists a sufficiently small $\varepsilon_0>0$ such that for every $\varphi\in(0,\varepsilon_0)$, 
	\[
	u(\varphi)>0\quad\text{and}\quad0<\frac{w(\varphi)}{u(\varphi)s(\varphi)}<1.
	\]
	This will be explained in detail in Section~\ref{proofofremark3}.
\end{remark}

Note that
\begin{align}\label{cal-5}
w(\varphi)-u(\varphi)s(\varphi)s(\theta)=\left\{
\begin{array}{cl}
	w(\varphi)&\text{ if $u(\varphi)=0$}
	\vspace{2mm}
	\\
	u(\varphi)s(\varphi)\big[\frac{w(\varphi)}{u(\varphi)s(\varphi)}-s(\theta)\big]&\text{ if $u(\varphi)\neq0$}
\end{array}\right..
\end{align}
Since $s(\varphi)\in(0,1)$ for all $\varphi\in(0,\pi)$, we have $s(\theta)\in(0,1)$ in particular. Then, by \eqref{cal-5}, it follows from Claim~\ref{clm-anaT'} that for any $\varphi\in(0,\pi)$,
\[
w(\varphi)-u(\varphi)s(\varphi)s(\theta)\neq0.
\]
Therefore, by the expression of $T_{\theta}'(\varphi)$ in \eqref{cal-sim-T'}, the equation $T_{\theta}'(\varphi)=0$ has a unique solution $\varphi=\theta$ in $(0,\pi)$.

\subsubsection{Conclusion: monotonicity of $T_{\theta}(\varphi)$}
Note that $k(0^+)=0$ and $l(\pi^-)=0$ by \eqref{def-klrs}. Hence by \eqref{rew-T}, we have
\[
T_{\theta}(0^+)=T_{\theta}(\pi^-)=0.
\]
In addition, the first formula in \eqref{exp-TT'} together with \eqref{expre-f-MA} implies that
\[
T_{\theta}(\theta)=e^{2\mathrm{Re}f_{\theta}(\theta)}=\Bigg(\frac{\sin\frac{\pi-\theta}{1+\alpha}}{\sin\frac{\pi-\theta}{1+\frac{1}{\alpha}}}\Bigg)^2>0.
\]
Now that $T_{\theta}'(\varphi)=0$ has a unique solution $\varphi=\theta$ in $(0,\pi)$, we can conclude that $T_{\theta}'(\varphi)>0$ on $(0,\theta)$ and $T_{\theta}'(\varphi)<0$ on $(\theta,\pi)$. Therefore, the function $T_{\theta}(\varphi)$ is strictly increasing on $(0,\theta)$ and strictly decreasing on $(\theta,\pi)$.

Thus far, the only remaining step in the proof of Lemma~\ref{lem-mono-T} is to prove Claim~\ref{clm-anaT'}, which will be provided in Section~\ref{prf-clm-anaT'}.

\section{Estimates of path integrals}
In this section, we estimate the path integrals $\mathcal{I}_{\mathrm{center}}(\theta,n)$ and $\mathcal{I}_{\mathrm{left}}(\theta,n)+\mathcal{I}_{\mathrm{right}}(\theta,n)$ defined in \eqref{def-int-cen} and \eqref{def-int-r+l}. Specifically, we prove Lemmas~\ref{est-int-cen} and \ref{est-int-r+l}. Throughout the proofs, we fix some $0<\varepsilon<\frac{\pi}{2}$ and consider $\theta\in[\varepsilon,\pi-\varepsilon]$. Furthermore, the parameter $\delta$ appearing in Section~\ref{Sec-del} satisfies $\delta\in(0,1/6)$.

\subsection{Proof of Lemma~\ref{est-int-cen}}\label{prf-est-int-cen-lem}
Recall the integral $\mathcal{I}_{\mathrm{center}}(\theta,n)$ defined in \eqref{def-int-cen}:
\begin{align}\label{copy-int-cent}
	\mathcal{I}_{\mathrm{center}}(\theta,n)=\frac{1}{\pi i}\int_{\theta-n^{-\frac{1}{2}+\delta}}^{\theta+n^{-\frac{1}{2}+\delta}}e^{nf_{\theta}(\varphi)}g_{\theta}(\varphi)\mathrm{d}\varphi.
\end{align}
By Lemma~\ref{lem-saddlepoint}, we have $f_{\theta}'(\theta)=0$. Hence, expanding the function $f_{\theta}(\varphi)$ at $\varphi=\theta$, we have
\begin{align}\label{expan-f}
	f_{\theta}(\varphi)=f_{\theta}(\theta)+\frac{f_{\theta}''(\theta)}{2}(\varphi-\theta)^2+\frac{f_{\theta}'''(\theta)}{6}(\varphi-\theta)^3+A_{\theta,\varphi}(\varphi-\theta)^4,
\end{align}
where
\[
A_{\theta,\varphi}:=\frac{\mathrm{Re}f^{(4)}_{\theta}(\phi_{1,\theta,\varphi})+i\mathrm{Im}f^{(4)}_{\theta}(\phi_{2,\theta,\varphi})}{24}\quad\text{with}\quad\phi_{1,\theta,\varphi},\,\phi_{2,\theta,\varphi}\in(\theta,\varphi)\text{ or }(\varphi,\theta).
\]
Similarly, expanding the function $g_{\theta}(\varphi)$ at $\varphi=\theta$, we have
\begin{align}\label{expan-g}
	g_{\theta}(\varphi)=g_{\theta}(\theta)+g_{\theta}'(\theta)(\varphi-\theta)+B_{\theta,\varphi}(\varphi-\theta)^2,
\end{align}
where
\[
B_{\theta,\varphi}:=\frac{\mathrm{Re}g^{''}_{\theta}(\phi_{3,\theta,\varphi})+i\mathrm{Im}g^{''}_{\theta}(\phi_{4,\theta,\varphi})}{2}\quad\text{with}\quad\phi_{3,\theta,\varphi},\,\phi_{4,\theta,\varphi}\in(\theta,\varphi)\text{ or }(\varphi,\theta).
\]
Thus by \eqref{copy-int-cent}, \eqref{expan-f} and \eqref{expan-g}, we obtain
\begin{align*}
	\mathcal{I}_{\mathrm{center}}(\theta,n)=\frac{1}{\pi i}\int_{\theta-n^{-\frac{1}{2}+\delta}}^{\theta+n^{-\frac{1}{2}+\delta}}&e^{n[f_{\theta}(\theta)+\frac{f_{\theta}''(\theta)}{2}(\varphi-\theta)^2+\frac{f_{\theta}'''(\theta)}{6}(\varphi-\theta)^3+A_{\theta,\varphi}(\varphi-\theta)^4]}\\
	&\times\big[g_{\theta}(\theta)+g_{\theta}'(\theta)(\varphi-\theta)+B_{\theta,\varphi}(\varphi-\theta)^2\big]\mathrm{d}\varphi.
\end{align*}
This can be rewritten as
\begin{align*}
	\mathcal{I}_{\mathrm{center}}(\theta,n)=\frac{1}{\pi i}g_{\theta}(\theta)e^{nf_{\theta}(\theta)}\int_{\theta-n^{-\frac{1}{2}+\delta}}^{\theta+n^{-\frac{1}{2}+\delta}}&e^{\frac{f_{\theta}''(\theta)}{2}n(\varphi-\theta)^2}\times e^{\frac{f_{\theta}'''(\theta)}{6}n(\varphi-\theta)^3+A_{\theta,\varphi}n(\varphi-\theta)^4}\\
	&\times\Big[1+\frac{g_{\theta}'(\theta)}{g_{\theta}(\theta)}(\varphi-\theta)+\frac{B_{\theta,\varphi}}{g_{\theta}(\theta)}(\varphi-\theta)^2\Big]\mathrm{d}\varphi.
\end{align*}
Making the change of variable $\varphi=\theta+n^{-\frac{1}{2}}\rho$, we obtain
\begin{align}\label{int-sec-cent}
	\begin{split}
		\mathcal{I}_{\mathrm{center}}(\theta,n)=\frac{1}{\pi i}n^{-\frac{1}{2}}g_{\theta}(\theta)e^{nf_{\theta}(\theta)}\int_{-n^{\delta}}^{n^{\delta}}&e^{\frac{f_{\theta}''(\theta)}{2}\rho^2}\times e^{\frac{f_{\theta}'''(\theta)}{6}n^{-\frac{1}{2}}\rho^3+A_{\theta,\rho}n^{-1}\rho^4}\\
		&\times\Big[1+\frac{g_{\theta}'(\theta)}{g_{\theta}(\theta)}n^{-\frac{1}{2}}\rho+\frac{B_{\theta,\rho}}{g_{\theta}(\theta)}n^{-1}\rho^2\Big]\mathrm{d}\rho.
	\end{split}
\end{align}
Here we simply denote $A_{\theta,\rho}=A_{\theta,\theta+n^{-\frac{1}{2}}\rho}$ and $B_{\theta,\rho}=B_{\theta,\theta+n^{-\frac{1}{2}}\rho}$.

Expanding the second factor of the integrand in \eqref{int-sec-cent} as follows:
\[
e^{\frac{f_{\theta}'''(\theta)}{6}n^{-\frac{1}{2}}\rho^3+A_{\theta,\rho}n^{-1}\rho^4}=1+\frac{f_{\theta}'''(\theta)}{6}n^{-\frac{1}{2}}\rho^3+A_{\theta,\rho}n^{-1}\rho^4+\sum_{j=2}^{\infty}\frac{\big[\frac{f_{\theta}'''(\theta)}{6}n^{-\frac{1}{2}}\rho^3+A_{\theta,\rho}n^{-1}\rho^4\big]^j}{j!}.
\]
Then, by a straightforward calculation and simplification, we obtain
\begin{align*}
	&\quad e^{\frac{f_{\theta}'''(\theta)}{6}n^{-\frac{1}{2}}\rho^3+A_{\theta,\rho}n^{-1}\rho^4}\times\Big[1+\frac{g_{\theta}'(\theta)}{g_{\theta}(\theta)}n^{-\frac{1}{2}}\rho+\frac{B_{\theta,\rho}}{g_{\theta}(\theta)}n^{-1}\rho^2\Big]\\
	&=1+\frac{g_{\theta}'(\theta)}{g_{\theta}(\theta)}n^{-\frac{1}{2}}\rho+\frac{f_{\theta}'''(\theta)}{6}n^{-\frac{1}{2}}\rho^3+L(\theta,n,\rho),
\end{align*}
where
\begin{align}\label{def-L}
	\begin{split}
		L(\theta,n,\rho):&=\frac{B_{\theta,\rho}}{g_{\theta}(\theta)}n^{-1}\rho^2+\frac{f_{\theta}'''(\theta)}{6}n^{-\frac{1}{2}}\rho^3\Big[\frac{g_{\theta}'(\theta)}{g_{\theta}(\theta)}n^{-\frac{1}{2}}\rho+\frac{B_{\theta,\rho}}{g_{\theta}(\theta)}n^{-1}\rho^2\Big]\\
		&+\Big(A_{\theta,\rho}n^{-1}\rho^4+\sum_{j=2}^{\infty}\frac{\big[\frac{f_{\theta}'''(\theta)}{6}n^{-\frac{1}{2}}\rho^3+A_{\theta,\rho}n^{-1}\rho^4\big]^j}{j!}\Big)\Big[1+\frac{g_{\theta}'(\theta)}{g_{\theta}(\theta)}n^{-\frac{1}{2}}\rho+\frac{B_{\theta,\rho}}{g_{\theta}(\theta)}n^{-1}\rho^2\Big].
	\end{split}
\end{align}
Therefore, we rewrite the integral \eqref{int-sec-cent} as
\begin{align}\label{int-thi-cent}
	\mathcal{I}_{\mathrm{center}}(\theta,n)=\frac{1}{\pi i}n^{-\frac{1}{2}}g_{\theta}(\theta)e^{nf_{\theta}(\theta)}\big[I_1(\theta,n)+I_2(\theta,n)+I_3(\theta,n)\big],
\end{align}
where
\begin{align}\label{def-I1}
	I_1(\theta,n):=\int_{-n^{\delta}}^{n^{\delta}}e^{\frac{f_{\theta}''(\theta)}{2}\rho^2}\mathrm{d}\rho,
\end{align}
\begin{align}\label{def-I2}
	I_2(\theta,n):=\int_{-n^{\delta}}^{n^{\delta}}e^{\frac{f_{\theta}''(\theta)}{2}\rho^2}\times\Big[\frac{g_{\theta}'(\theta)}{g_{\theta}(\theta)}n^{-\frac{1}{2}}\rho+\frac{f_{\theta}'''(\theta)}{6}n^{-\frac{1}{2}}\rho^3\Big]\mathrm{d}\rho
\end{align}
and
\begin{align}\label{def-I3}
	I_3(\theta,n):=\int_{-n^{\delta}}^{n^{\delta}}e^{\frac{f_{\theta}''(\theta)}{2}\rho^2}\times L(\theta,n,\rho)\mathrm{d}\rho.
\end{align}

\subsubsection{Estimate of $I_1(\theta,n)$}
Rewrite the integral $I_1(\theta,n)$ defined in \eqref{def-I1} as
\begin{align}\label{rewriteI1}
	I_1(\theta,n)=\int_{-\infty}^{+\infty}e^{\frac{f_{\theta}''(\theta)}{2}\rho^2}\mathrm{d}\rho-2\int_{n^{\delta}}^{+\infty}e^{\frac{f_{\theta}''(\theta)}{2}\rho^2}\mathrm{d}\rho.
\end{align}
Note that
\[
\int_{-\infty}^{+\infty}e^{\frac{f_{\theta}''(\theta)}{2}\rho^2}\mathrm{d}\rho=\Big(\int_{\mathbb{R}^2}e^{\frac{f_{\theta}''(\theta)}{2}(x^2+y^2)}\mathrm{d}x\mathrm{d}y\Big)^{\frac{1}{2}}=\Big(2\pi\int_{0}^{+\infty}e^{\frac{f_{\theta}''(\theta)}{2}r^2}r\mathrm{d}r\Big)^{\frac{1}{2}}.
\]
Since $\mathrm{Re}f_{\theta}''(\theta)<0$ by Corollary~\ref{Cor-secderiv}, we have $\lim_{r\to+\infty}e^{\frac{f_{\theta}''(\theta)}{2}r^2}=0$ and hence
\[
\int_{0}^{+\infty}e^{\frac{f_{\theta}''(\theta)}{2}r^2}r\mathrm{d}r=\frac{1}{f_{\theta}''(\theta)}e^{\frac{f_{\theta}''(\theta)}{2}r^2}\Big|_{r=0}^{r=+\infty}=0-\frac{1}{f_{\theta}''(\theta)}=-\frac{1}{f_{\theta}''(\theta)}.
\]
This implies that
\begin{align}\label{cal-rewriteI1}
	\int_{-\infty}^{+\infty}e^{\frac{f_{\theta}''(\theta)}{2}\rho^2}\mathrm{d}\rho=i\sqrt{\frac{2\pi}{f_{\theta}''(\theta)}}.
\end{align}
Moreover, it follows from $\mathrm{Re}f_{\theta}''(\theta)<0$ that
\begin{align}\label{cal-rewriteI1remain}
	\Big|\int_{n^{\delta}}^{+\infty}e^{\frac{f_{\theta}''(\theta)}{2}\rho^2}\mathrm{d}\rho\Big|\leq\int_{n^{\delta}}^{+\infty}e^{\frac{\mathrm{Re}f_{\theta}''(\theta)}{2}\rho^2}\mathrm{d}\rho\leq\frac{1}{n^{\delta}}\int_{n^{\delta}}^{+\infty}e^{\frac{\mathrm{Re}f_{\theta}''(\theta)}{2}\rho^2}\rho\mathrm{d}\rho=O\Big(\frac{e^{\frac{\mathrm{Re}f_{\theta}''(\theta)}{2}n^{2\delta}}}{n^{\delta}}\Big)=O(n^{-1}).
\end{align}
Hence by \eqref{rewriteI1}, \eqref{cal-rewriteI1} and \eqref{cal-rewriteI1remain}, we obtain
\begin{align}\label{resultI1}
	I_1(\theta,n)=i\sqrt{\frac{2\pi}{f_{\theta}''(\theta)}}+O(n^{-1}).
\end{align}

\subsubsection{Estimate of $I_2(\theta,n)$}
Noticing that the integrand in the integral $I_2(\theta,n)$ defined in \eqref{def-I2} is odd, we directly have
\begin{align}\label{resultI2}
	I_2(\theta,n)=0.
\end{align}

\subsubsection{Estimate of $I_3(\theta,n)$}
Recall the definition of $I_3(\theta,n)$ in \eqref{def-I3}:
\begin{align}\label{rewriteI3}
	I_3(\theta,n)=\int_{-n^{\delta}}^{n^{\delta}}e^{\frac{f_{\theta}''(\theta)}{2}\rho^2}\times L(\theta,n,\rho)\mathrm{d}\rho
\end{align}
with $L(\theta,n,\rho)$ defined in \eqref{def-L}. Note that
\begin{align*}
	&\quad\sum_{j=2}^{\infty}\frac{\big[\frac{f_{\theta}'''(\theta)}{6}n^{-\frac{1}{2}}\rho^3+A_{\theta,\rho}n^{-1}\rho^4\big]^j}{j!}\\
	&=\Big[\frac{f_{\theta}'''(\theta)}{6}n^{-\frac{1}{2}}\rho^3+A_{\theta,\rho}n^{-1}\rho^4\Big]^2\times\sum_{j=0}^{\infty}\frac{\big[\frac{f_{\theta}'''(\theta)}{6}n^{-\frac{1}{2}}\rho^3+A_{\theta,\rho}n^{-1}\rho^4\big]^{j}}{(j+2)!}
\end{align*}
Since $\delta\in(0,1/6)$ and $\rho\in(-n^{\delta},n^{\delta})$, the function
$\frac{f_{\theta}'''(\theta)}{6}n^{-\frac{1}{2}}\rho^3+A_{\theta,\rho}n^{-1}\rho^4$ is uniformly bounded for $\theta\in[\varepsilon,\pi-\varepsilon]$ with any fixed $0<\varepsilon<\frac{\pi}{2}$. Hence the function
\[
\sum_{j=0}^{\infty}\frac{\big[\frac{f_{\theta}'''(\theta)}{6}n^{-\frac{1}{2}}\rho^3+A_{\theta,\rho}n^{-1}\rho^4\big]^{j}}{(j+2)!}
\]
is uniformly bounded, which implies that
\begin{align*}
	\sum_{j=2}^{\infty}\frac{\big[\frac{f_{\theta}'''(\theta)}{6}n^{-\frac{1}{2}}\rho^3+A_{\theta,\rho}n^{-1}\rho^4\big]^j}{j!}&=\Big[\frac{f_{\theta}'''(\theta)}{6}n^{-\frac{1}{2}}\rho^3+A_{\theta,\rho}n^{-1}\rho^4\Big]^2\times O(1)\\
	&=(|\rho|^6+|\rho|^7+|\rho|^8)\times O(n^{-1}).
\end{align*}
Therefore, from the expression of $L(\theta,n,\rho)$ in \eqref{def-L}, one can show that
\begin{align}\label{rewriteL}
	L(\theta,n,\rho)=(|\rho|^2+|\rho|^4+|\rho|^5+|\rho|^6+|\rho|^7+|\rho|^8+|\rho|^9+|\rho|^{10})\times O(n^{-1}).
\end{align}
Note that the above analysis reveals why the parameter $\delta$ cannot exceed $1/6$. It follows from \eqref{rewriteI3} and \eqref{rewriteL} that
\[
I_3(\theta,n)=\int_{-n^{\delta}}^{n^{\delta}}e^{\frac{f_{\theta}''(\theta)}{2}\rho^2}(|\rho|^2+|\rho|^4+|\rho|^5+|\rho|^6+|\rho|^7+|\rho|^8+|\rho|^9+|\rho|^{10})\mathrm{d}\rho\times O(n^{-1}).
\]
Then, since $\mathrm{Re}f_{\theta}''(\theta)<0$ by Corollary~\ref{Cor-secderiv}, we have
\begin{align}\label{resultI3}
	I_3(\theta,n)=O(1)\times O(n^{-1})=O(n^{-1}).
\end{align}

\subsubsection{Conclusion: estimate of $\mathcal{I}_{\mathrm{center}}(\theta,n)$}
Combining \eqref{int-thi-cent} with \eqref{resultI1}, \eqref{resultI2} and \eqref{resultI3}, we obtain
\begin{align*}
	\mathcal{I}_{\mathrm{center}}(\theta,n)&=\frac{1}{\pi i}n^{-\frac{1}{2}}g_{\theta}(\theta)e^{nf_{\theta}(\theta)}\Big[i\sqrt{\frac{2\pi}{f_{\theta}''(\theta)}}+O(n^{-1})+0+O(n^{-1})\Big]\\
	&=\sqrt{2}\pi^{-\frac{1}{2}}n^{-\frac{1}{2}}\frac{g_{\theta}(\theta)}{\sqrt{f_{\theta}''(\theta)}}e^{nf_{\theta}(\theta)}\big[1+O(n^{-1})\big].
\end{align*}
This completes the proof of Lemma~\ref{est-int-cen}.

\subsection{Proof of Lemma~\ref{est-int-r+l}}\label{prf-est-int-r+l-lem}
Without loss of generality, we estimate only the integral $\mathcal{I}_{\mathrm{left}}(\theta,n)$, since the treatment of $\mathcal{I}_{\mathrm{right}}(\theta,n)$ is similar. Recall the definition of $\mathcal{I}_{\mathrm{left}}(\theta,n)$ from the first formula in \eqref{def-int-r+l}:
\begin{align}\label{copy-int-left}
	\mathcal{I}_{\mathrm{left}}(\theta,n)=\frac{1}{\pi i}\int_{0}^{\theta-n^{-\frac{1}{2}+\delta}}e^{nf_{\theta}(\varphi)}g_{\theta}(\varphi)\mathrm{d}\varphi.
\end{align}
Noticing the definition of $T_{\theta}(\varphi)$ in \eqref{def-T}, it follows from \eqref{copy-int-left} that
\[
|\mathcal{I}_{\mathrm{left}}(\theta,n)|\leq\frac{1}{\pi}\int_{0}^{\theta-n^{-\frac{1}{2}+\delta}}\big|e^{nf_{\theta}(\varphi)}\big|\cdot|g_{\theta}(\varphi)|\mathrm{d}\varphi=\frac{1}{\pi}\int_{0}^{\theta-n^{-\frac{1}{2}+\delta}}\big[T_{\theta}(\varphi)\big]^{\frac{n}{2}}\cdot|g_{\theta}(\varphi)|\mathrm{d}\varphi.
\]
Since $T_{\theta}(\varphi)$ is strictly increasing on $(0,\theta-n^{-\frac{1}{2}+\delta})$ by Lemma~\ref{lem-mono-T}, we obtain
\begin{align}\label{est-T-Ileft}
	\begin{split}
		|\mathcal{I}_{\mathrm{left}}(\theta,n)|&\leq\big[T_{\theta}(\theta-n^{-\frac{1}{2}+\delta})\big]^{\frac{n}{2}}\frac{1}{\pi}\int_{0}^{\theta-n^{-\frac{1}{2}+\delta}}|g_{\theta}(\varphi)|\mathrm{d}\varphi\\
		&\leq\big[T_{\theta}(\theta-n^{-\frac{1}{2}+\delta})\big]^{\frac{n}{2}}\frac{1}{\pi}\int_{0}^{\theta}|g_{\theta}(\varphi)|\mathrm{d}\varphi\\
		&=\big[T_{\theta}(\theta-n^{-\frac{1}{2}+\delta})\big]^{\frac{n}{2}}\times O(1).
	\end{split}
\end{align}

Using again the definition of $T_{\theta}(\varphi)$ in \eqref{def-T}, we have
\begin{align}\label{usingagaindefT}
	\big[T_{\theta}(\theta-n^{-\frac{1}{2}+\delta})\big]^{\frac{n}{2}}=e^{n\mathrm{Re}f_{\theta}(\theta-n^{-\frac{1}{2}+\delta})}.
\end{align}
By Lemma~\ref{lem-saddlepoint}, $\mathrm{Re}f_{\theta}'(\theta)=0$. Expanding $\mathrm{Re}f_{\theta}(\theta-n^{-\frac{1}{2}+\delta})$ around $\theta$ gives
\begin{align}\label{expan-f-again}
	\mathrm{Re}f_{\theta}(\theta-n^{-\frac{1}{2}+\delta})=\mathrm{Re}f_{\theta}(\theta)+\frac{\mathrm{Re}f_{\theta}''(\theta)}{2}n^{-1+2\delta}-\frac{\mathrm{Re}f_{\theta}'''(\phi_{\theta,n})}{6}n^{-\frac{3}{2}+3\delta}
\end{align}
with $\phi_{\theta,n}\in(\theta-n^{-\frac{1}{2}+\delta},\theta)$. Hence, by \eqref{usingagaindefT} and \eqref{expan-f-again}, we obtain
\[
\big[T_{\theta}(\theta-n^{-\frac{1}{2}+\delta})\big]^{\frac{n}{2}}=e^{n\mathrm{Re}f_{\theta}(\theta)}\cdot e^{\frac{\mathrm{Re}f_{\theta}''(\theta)}{2}n^{2\delta}}\cdot e^{-\frac{\mathrm{Re}f_{\theta}'''(\phi_{\theta,n})}{6}n^{-\frac{1}{2}+3\delta}}.
\]
Since $\delta\in(0,1/6)$ and $\mathrm{Re}f_{\theta}'''(\phi_{\theta,n})$ is uniformly bounded as $n\to\infty$ for $\theta\in[\varepsilon,\pi-\varepsilon]$ with any fixed $0<\varepsilon<\frac{\pi}{2}$, we have
\[
e^{-\frac{\mathrm{Re}f_{\theta}'''(\phi_{\theta,n})}{6}n^{-\frac{1}{2}+3\delta}}=O(1).
\]
Note that the above analysis also reveals why the parameter $\delta$ cannot exceed $1/6$. This yields
\[
\big[T_{\theta}(\theta-n^{-\frac{1}{2}+\delta})\big]^{\frac{n}{2}}=e^{n\mathrm{Re}f_{\theta}(\theta)}\cdot e^{\frac{\mathrm{Re}f_{\theta}''(\theta)}{2}n^{2\delta}}\cdot O(1)=n^{-\frac{1}{2}}e^{nf_{\theta}(\theta)}\cdot O\Big(n^{\frac{1}{2}}e^{\frac{\mathrm{Re}f_{\theta}''(\theta)}{2}n^{2\delta}}\Big).
\]
Then, since $\mathrm{Re}f_{\theta}''(\theta)<0$ by Corollary~\ref{Cor-secderiv}, we have
\begin{align}\label{final-T}
	\big[T_{\theta}(\theta-n^{-\frac{1}{2}+\delta})\big]^{\frac{n}{2}}=n^{-\frac{1}{2}}e^{nf_{\theta}(\theta)}\cdot o(n^{-1})=\sqrt{2}\pi^{-\frac{1}{2}}n^{-\frac{1}{2}}\frac{g_{\theta}(\theta)}{\sqrt{f_{\theta}''(\theta)}}e^{nf_{\theta}(\theta)}\cdot o(n^{-1}).
\end{align}

Finally, by \eqref{est-T-Ileft} and \eqref{final-T}, we obtain
\[
\mathcal{I}_{\mathrm{left}}(\theta,n)=\sqrt{2}\pi^{-\frac{1}{2}}n^{-\frac{1}{2}}\frac{g_{\theta}(\theta)}{\sqrt{f_{\theta}''(\theta)}}e^{nf_{\theta}(\theta)}\cdot o(n^{-1}).
\]
Similarly, one can show that
\[
\mathcal{I}_{\mathrm{right}}(\theta,n)=\sqrt{2}\pi^{-\frac{1}{2}}n^{-\frac{1}{2}}\frac{g_{\theta}(\theta)}{\sqrt{f_{\theta}''(\theta)}}e^{nf_{\theta}(\theta)}\cdot o(n^{-1}).
\]
This completes the proof of Lemma~\ref{est-int-r+l}.

\section{Proof of Claim~\ref{clm-anaT'}}\label{prf-clm-anaT'}
In this section, we are going to prove Claim~\ref{clm-anaT'}. Namely, we will show that $w(\varphi)\neq0$ when $u(\varphi)=0$, and that $\frac{w(\varphi)}{u(\varphi)s(\varphi)}\notin(0,1)$ when $u(\varphi)\neq0$.

Recall the definitions of $u(\varphi)$ and $w(\varphi)$ in \eqref{def-uvw}:
\begin{align}\label{redef-uw}
	u(\varphi)=(kl)'(\varphi)=k(\varphi)l'(\varphi)+k'(\varphi)l(\varphi)\quad\text{and}\quad w(\varphi)=(k^2l')(\varphi)=k^2(\varphi)l'(\varphi)
\end{align}
with $k(\varphi)$ and $l(\varphi)$ defined in \eqref{def-klrs}:
\begin{align}\label{redef-kl}
	k(\varphi)=\big[\Theta_{\frac{1}{\alpha}}(\varphi)\big]^{2\alpha}\quad\text{and}\quad l(\varphi)=1+\big[\Theta_{\frac{1}{\alpha}}(\varphi)\big]^2-2\Theta_{\frac{1}{\alpha}}(\varphi)\cos\frac{\pi-\varphi}{1+\alpha}.
\end{align}
Also recall the definition of $s(\varphi)$ from the last formula in \eqref{def-klrs}:
\begin{align}\label{redef-s}
	s(\varphi)=\big[\Theta_{\frac{1}{\alpha}}(\varphi)\big]^{\alpha}\,\Theta_{\alpha}(\varphi).
\end{align}
Here we repeat once again the definitions of 
$\Theta_{\frac{1}{\alpha}}(\varphi)$ and $\Theta_{\alpha}(\varphi)$:
\begin{align}\label{redef-Tha}
	\Theta_{\frac{1}{\alpha}}(\varphi)=\frac{\sin\varphi}{(1+\frac{1}{\alpha})\sin\frac{\pi-\varphi}{1+\frac{1}{\alpha}}}\quad\text{and}\quad\Theta_{\alpha}(\varphi)=\frac{\sin\varphi}{(1+\alpha)\sin\frac{\pi-\varphi}{1+\alpha}}.
\end{align}

\subsection{Auxiliary function $d(\varphi)$}\label{Sec-Auxilifunc-d}
To facilitate the subsequent expressions for $u(\varphi)$ and $w(\varphi)$, we first introduce an auxiliary function 
$d(\varphi)$:
\begin{align}\label{def-d}
	d(\varphi):=\frac{(1+\alpha)\cos\varphi\sin\frac{\pi-\varphi}{1+\frac{1}{\alpha}}+\alpha\sin\varphi\cos\frac{\pi-\varphi}{1+\frac{1}{\alpha}}}{\sin\varphi\sin\frac{\pi-\varphi}{1+\frac{1}{\alpha}}},\quad\varphi\in(0,\pi).
\end{align}

\begin{lemma}\label{lem-d}
	For any fixed $\alpha>0$, the function $d(\varphi)$ is strictly decreasing on $(0,\pi)$ with \[d(0^+)=+\infty\quad\text{and}\quad d(\pi^-)=0.\] Consequently, there exists a unique $\varphi_0\in(0,\pi)$ such that $d(\varphi_0)=1$.
\end{lemma}

\begin{proof}
Rewrite $d(\varphi)$ defined in \eqref{def-d} in the form
\[
d(\varphi)=(1+\alpha)\cot\varphi+\alpha\cot\frac{\pi-\varphi}{1+\frac{1}{\alpha}}.
\]
Then the derivative of $d(\varphi)$ is
\[
d'(\varphi)=-(1+\alpha)\csc^2\varphi+\frac{\alpha^2}{1+\alpha}\csc^2\frac{\pi-\varphi}{1+\frac{1}{\alpha}}=-\frac{(1+\alpha)^2\sin^2\frac{\pi-\varphi}{1+\frac{1}{\alpha}}-\alpha^2\sin^2\varphi}{(1+\alpha)\sin^2\varphi\sin^2\frac{\pi-\varphi}{1+\frac{1}{\alpha}}},
\]
which can be rearranged as
\begin{align}\label{simpli-iden4}
	d'(\varphi)=-\frac{1+\alpha}{\sin^2\varphi}\Big[1-\Big(\frac{\sin\varphi}{(1+\frac{1}{\alpha})\sin\frac{\pi-\varphi}{1+\frac{1}{\alpha}}}\Big)^2\Big]=-\frac{1+\alpha}{\sin^2\varphi}\Big(1-\big[\Theta_{\frac{1}{\alpha}}(\varphi)\big]^{2}\Big).
\end{align}
It follows from the proof of Lemma~\ref{lem-xtha} that $\Theta_{\frac{1}{\alpha}}(\varphi)\in(0,1)$ for all $\varphi\in(0,\pi)$. Hence $d'(\varphi)<0$ for all $\varphi\in(0,\pi)$. Thus the function $d(\varphi)$ is strictly decreasing on $(0,\pi)$. In addition, the facts $d(0^+)=+\infty$ and $d(\pi^-)=0$ can be easily verified from \eqref{def-d} by L'Hôpital's rule.
\end{proof}

\subsection{Expression for $u(\varphi)$}
By \eqref{redef-kl}, we have
\[
k'(\varphi)=2\alpha\big[\Theta_{\frac{1}{\alpha}}(\varphi)\big]^{2\alpha-1}\,\Theta_{\frac{1}{\alpha}}'(\varphi)
\]
and
\begin{align}\label{cal-l'}
	l'(\varphi)=2\Theta_{\frac{1}{\alpha}}'(\varphi)\Big[\Theta_{\frac{1}{\alpha}}(\varphi)-\cos\frac{\pi-\varphi}{1+\alpha}\Big]-\frac{2}{1+\alpha}\Theta_{\frac{1}{\alpha}}(\varphi)\sin\frac{\pi-\varphi}{1+\alpha}.
\end{align}
Hence, after a simple calculation and simplification for the first formula in \eqref{redef-uw}, we get
\begin{align}\label{cal-u}
	\begin{split}
		u(\varphi)&=2\big[\Theta_{\frac{1}{\alpha}}(\varphi)\big]^{2\alpha-1}\,\Theta_{\frac{1}{\alpha}}'(\varphi)\Big\{(1+\alpha)\big[\Theta_{\frac{1}{\alpha}}(\varphi)\big]^{2}-(1+2\alpha)\Theta_{\frac{1}{\alpha}}(\varphi)\cos\frac{\pi-\varphi}{1+\alpha}+\alpha\Big\}\\
		&\quad-\frac{2}{1+\alpha}\big[\Theta_{\frac{1}{\alpha}}(\varphi)\big]^{2\alpha+1}\sin\frac{\pi-\varphi}{1+\alpha}.
	\end{split}
\end{align}
We shall make use of the following simplification identity:
\begin{align}\label{simpli-iden}
	(1+\alpha)\big[\Theta_{\frac{1}{\alpha}}(\varphi)\big]^{2}-(1+2\alpha)\Theta_{\frac{1}{\alpha}}(\varphi)\cos\frac{\pi-\varphi}{1+\alpha}+\alpha=(1+\alpha)\Theta_{\frac{1}{\alpha}}'(\varphi)\sin\frac{\pi-\varphi}{1+\alpha}.
\end{align}
Here we leave the verification of identity \eqref{simpli-iden} to Appendix~\ref{prf-simpli-iden}. Thus by \eqref{cal-u} and identity \eqref{simpli-iden}, we obtain
\begin{align}\label{cal-uagain}
	\begin{split}
		u(\varphi)&=2(1+\alpha)\big[\Theta_{\frac{1}{\alpha}}(\varphi)\big]^{2\alpha-1}\big[\Theta_{\frac{1}{\alpha}}'(\varphi)\big]^2\sin\frac{\pi-\varphi}{1+\alpha}-\frac{2}{1+\alpha}\big[\Theta_{\frac{1}{\alpha}}(\varphi)\big]^{2\alpha+1}\sin\frac{\pi-\varphi}{1+\alpha}\\
		&=\frac{2\sin\frac{\pi-\varphi}{1+\alpha}}{1+\alpha}\big[\Theta_{\frac{1}{\alpha}}(\varphi)\big]^{2\alpha+1}\Bigg\{\frac{(1+\alpha)^2[\Theta_{\frac{1}{\alpha}}'(\varphi)]^2}{[\Theta_{\frac{1}{\alpha}}(\varphi)]^2}-1\Bigg\}.
	\end{split}
\end{align}

Similar to the simple calculation of $\Theta_{\alpha}'(\theta)$ in the proof of Lemma~\ref{lem-xtha}, one can show that
\begin{align}\label{cal-der-Tha-phi}
	\Theta_{\frac{1}{\alpha}}'(\varphi)=\frac{(1+\frac{1}{\alpha})\cos\varphi\sin\frac{\pi-\varphi}{1+\frac{1}{\alpha}}+\sin\varphi\cos\frac{\pi-\varphi}{1+\frac{1}{\alpha}}}{(1+\frac{1}{\alpha})^2\sin^2\frac{\pi-\varphi}{1+\frac{1}{\alpha}}}.
\end{align}
Then, by the expressions of $d(\varphi)$ in \eqref{def-d} and $\Theta_{\frac{1}{\alpha}}(\varphi)$ in \eqref{redef-Tha}, a simple simplification yields
\begin{align}\label{cal-Tha'}
	\Theta_{\frac{1}{\alpha}}'(\varphi)=\frac{d(\varphi)\Theta_{\frac{1}{\alpha}}(\varphi)}{1+\alpha}.
\end{align}
Hence by \eqref{cal-Tha'}, we can rewrite \eqref{cal-uagain} as
\begin{align}\label{exp-finalu}
	u(\varphi)=\frac{2\sin\frac{\pi-\varphi}{1+\alpha}}{1+\alpha}\big[\Theta_{\frac{1}{\alpha}}(\varphi)\big]^{2\alpha+1}\big[d^2(\varphi)-1\big].
\end{align}

\subsection{Expression for $w(\varphi)$}
For the second formula in \eqref{redef-uw}, by \eqref{redef-kl} and \eqref{cal-l'}, we have
\[
w(\varphi)=\big[\Theta_{\frac{1}{\alpha}}(\varphi)\big]^{4\alpha}\Big\{2\Theta_{\frac{1}{\alpha}}'(\varphi)\Big[\Theta_{\frac{1}{\alpha}}(\varphi)-\cos\frac{\pi-\varphi}{1+\alpha}\Big]-\frac{2}{1+\alpha}\Theta_{\frac{1}{\alpha}}(\varphi)\sin\frac{\pi-\varphi}{1+\alpha}\Big\}.
\]
Using \eqref{cal-Tha'} and after a simple simplification, we obtain
\begin{align}\label{cal-w}
	\begin{split}
		w(\varphi)&=\frac{2}{1+\alpha}\big[\Theta_{\frac{1}{\alpha}}(\varphi)\big]^{4\alpha+1}\Big\{d(\varphi)\Big[\Theta_{\frac{1}{\alpha}}(\varphi)-\cos\frac{\pi-\varphi}{1+\alpha}\Big]-\sin\frac{\pi-\varphi}{1+\alpha}\Big\}\\
		&=\frac{2\sin\varphi}{(1+\alpha)^2}\big[\Theta_{\frac{1}{\alpha}}(\varphi)\big]^{4\alpha+1}\Big\{d(\varphi)\frac{1+\alpha}{\sin\varphi}\Big[\Theta_{\frac{1}{\alpha}}(\varphi)-\cos\frac{\pi-\varphi}{1+\alpha}\Big]-\frac{(1+\alpha)\sin\frac{\pi-\varphi}{1+\alpha}}{\sin\varphi}\Big\}.
	\end{split}
\end{align}
The following two simplification identities will be used here:
\begin{align}\label{simpli-iden2}
	\frac{1+\alpha}{\sin\varphi}\Big[\Theta_{\frac{1}{\alpha}}(\varphi)-\cos\frac{\pi-\varphi}{1+\alpha}\Big]=d(\varphi)\cos\frac{\pi-\varphi}{1+\frac{1}{\alpha}}-\sin\frac{\pi-\varphi}{1+\frac{1}{\alpha}},
\end{align}
\begin{align}\label{simpli-iden3}
	\frac{(1+\alpha)\sin\frac{\pi-\varphi}{1+\alpha}}{\sin\varphi}=d(\varphi)\sin\frac{\pi-\varphi}{1+\frac{1}{\alpha}}+\cos\frac{\pi-\varphi}{1+\frac{1}{\alpha}}.
\end{align}
We leave the verification of identities \eqref{simpli-iden2} and \eqref{simpli-iden3} to Appendices~\ref{prf-simpli-iden2} and \ref{prf-simpli-iden3}, respectively. Now combining \eqref{cal-w} with identities \eqref{simpli-iden2} and \eqref{simpli-iden3}, after a simple calculation, we obtain
\begin{align}\label{exp-finalw}
	w(\varphi)=\frac{2\sin\varphi}{(1+\alpha)^2}\big[\Theta_{\frac{1}{\alpha}}(\varphi)\big]^{4\alpha+1}\Big\{\big[d^2(\varphi)-1\big]\cos\frac{\pi-\varphi}{1+\frac{1}{\alpha}}-2d(\varphi)\sin\frac{\pi-\varphi}{1+\frac{1}{\alpha}}\Big\}.
\end{align}

\subsection{The case $u(\varphi)=0$}\label{Sec-uvarphi0}
For the equation $u(\varphi)=0$, examining formula \eqref{exp-finalu}, we know that $d^2(\varphi)=1$. Since by Lemma~\ref{lem-d} we have $d(\varphi)\in(0,+\infty)$, it follows that $d(\varphi)=1$ and hence $\varphi=\varphi_0$ with $d(\varphi_0)=1$.

Therefore, in the case when $u(\varphi)=0$, by \eqref{exp-finalw}, we have
\[
w(\varphi)=w(\varphi_0)=-\frac{4\sin\varphi_0\sin\frac{\pi-\varphi_0}{1+\frac{1}{\alpha}}}{(1+\alpha)^2}\big[\Theta_{\frac{1}{\alpha}}(\varphi_0)\big]^{4\alpha+1}<0.
\]
This proves the first statement of Claim~\ref{clm-anaT'}.

\subsection{The case $u(\varphi)\neq0$}\label{Sec-uvarphineq0}
Since $u(\varphi)\neq0$, the above analysis in Section~\ref{Sec-uvarphi0} implies $\varphi\neq\varphi_0$. Furthermore, by Lemma~\ref{lem-d}, we have $d(\varphi)>1$ for $\varphi\in(0,\varphi_0)$ and $d(\varphi)<1$ for $\varphi\in(\varphi_0,\pi)$. Let us consider the function
\begin{align}\label{def-h}
	h(\varphi):=\frac{w(\varphi)}{u(\varphi)s(\varphi)},\quad\varphi\in(0,\varphi_0)\cup(\varphi_0,\pi).
\end{align}
We are going to show that $h(\varphi)\notin(0,1)$.

By \eqref{redef-s}, \eqref{exp-finalu} and \eqref{exp-finalw}, a simple calculation yields
\begin{align}\label{exp-h}
	h(\varphi)=\big[\Theta_{\frac{1}{\alpha}}(\varphi)\big]^{\alpha}\Big[\cos\frac{\pi-\varphi}{1+\frac{1}{\alpha}}-\frac{2d(\varphi)}{d^2(\varphi)-1}\sin\frac{\pi-\varphi}{1+\frac{1}{\alpha}}\Big],\quad\varphi\in(0,\varphi_0)\cup(\varphi_0,\pi).
\end{align}
Moreover, using the facts $\Theta_{\frac{1}{\alpha}}(0^+)=0$, $\Theta_{\frac{1}{\alpha}}(\pi^-)=1$, and noticing that
\[
\lim_{\varphi\to\varphi_0^-}\frac{1}{d^2(\varphi)-1}=+\infty,\quad \lim_{\varphi\to\varphi_0^+}\frac{1}{d^2(\varphi)-1}=-\infty,
\]
one can show that
\begin{align}\label{fac-h-0varphipi}
	h(0^+)=0,\quad h(\varphi_0^-)=-\infty,\quad h(\varphi_0^+)=+\infty,\quad h(\pi^-)=1.
\end{align}

\subsubsection{Expression for $h'(\varphi)$}
The derivative of $h(\varphi)$ is
\begin{align*}
	h&'(\varphi)=\alpha\big[\Theta_{\frac{1}{\alpha}}(\varphi)\big]^{\alpha-1}\,\Theta_{\frac{1}{\alpha}}'(\varphi)\Big[\cos\frac{\pi-\varphi}{1+\frac{1}{\alpha}}-\frac{2d(\varphi)}{d^2(\varphi)-1}\sin\frac{\pi-\varphi}{1+\frac{1}{\alpha}}\Big]\\
	&+\big[\Theta_{\frac{1}{\alpha}}(\varphi)\big]^{\alpha}\Big[\frac{\alpha}{1+\alpha}\sin\frac{\pi-\varphi}{1+\frac{1}{\alpha}}+\frac{\alpha}{1+\alpha}\cdot\frac{2d(\varphi)}{d^2(\varphi)-1}\cos\frac{\pi-\varphi}{1+\frac{1}{\alpha}}+\frac{2d'(\varphi)[d^2(\varphi)+1]}{[d^2(\varphi)-1]^2}\sin\frac{\pi-\varphi}{1+\frac{1}{\alpha}}\Big].
\end{align*}
Using the identity \eqref{cal-Tha'}:
\[
\Theta_{\frac{1}{\alpha}}'(\varphi)=\frac{d(\varphi)\Theta_{\frac{1}{\alpha}}(\varphi)}{1+\alpha},
\]
and by a simple rearrangement, we obtain
\[
h'(\varphi)=\frac{\alpha}{1+\alpha}\cdot\frac{d^2(\varphi)+1}{d^2(\varphi)-1}\big[\Theta_{\frac{1}{\alpha}}(\varphi)\big]^{\alpha}\Big[d(\varphi)\cos\frac{\pi-\varphi}{1+\frac{1}{\alpha}}-\sin\frac{\pi-\varphi}{1+\frac{1}{\alpha}}+\frac{1+\alpha}{\alpha}\cdot\frac{2d'(\varphi)}{d^2(\varphi)-1}\sin\frac{\pi-\varphi}{1+\frac{1}{\alpha}}\Big].
\]
Recall the identities \eqref{simpli-iden2} and \eqref{simpli-iden4}:
\[
d(\varphi)\cos\frac{\pi-\varphi}{1+\frac{1}{\alpha}}-\sin\frac{\pi-\varphi}{1+\frac{1}{\alpha}}=\frac{1+\alpha}{\sin\varphi}\Big[\Theta_{\frac{1}{\alpha}}(\varphi)-\cos\frac{\pi-\varphi}{1+\alpha}\Big]
\]
and
\[
d'(\varphi)=-\frac{1+\alpha}{\sin^2\varphi}\Big(1-\big[\Theta_{\frac{1}{\alpha}}(\varphi)\big]^{2}\Big),
\]
then one can show that
\begin{align}\label{exp-h'}
	h'(\varphi)=-\frac{\alpha}{\sin\varphi}\cdot\frac{d^2(\varphi)+1}{[d^2(\varphi)-1]^2}\big[\Theta_{\frac{1}{\alpha}}(\varphi)\big]^{\alpha-1}\Delta(\varphi)
\end{align}
with
\begin{align}\label{def-Del}
	\Delta(\varphi):=[d^2(\varphi)-1]\Theta_{\frac{1}{\alpha}}(\varphi)\Big[\cos\frac{\pi-\varphi}{1+\alpha}-\Theta_{\frac{1}{\alpha}}(\varphi)\Big]+2\Big(1-\big[\Theta_{\frac{1}{\alpha}}(\varphi)\big]^{2}\Big).
\end{align}

\subsubsection{Analysis of $h'(\varphi)$}\label{Analysisofhvarphi}
Note that
\[
1-\big[\Theta_{\frac{1}{\alpha}}(\varphi)\big]^{2}=\big[1+\Theta_{\frac{1}{\alpha}}(\varphi)\big]\big[1-\Theta_{\frac{1}{\alpha}}(\varphi)\big]\geq\big[1+\Theta_{\frac{1}{\alpha}}(\varphi)\big]\Big[\cos\frac{\pi-\varphi}{1+\alpha}-\Theta_{\frac{1}{\alpha}}(\varphi)\Big].
\]
Thus by \eqref{def-Del}, we obtain
\begin{align}\label{analy-Del}
	\Delta(\varphi)\geq\big[d^2(\varphi)\Theta_{\frac{1}{\alpha}}(\varphi)+\Theta_{\frac{1}{\alpha}}(\varphi)+2\big]\Big[\cos\frac{\pi-\varphi}{1+\alpha}-\Theta_{\frac{1}{\alpha}}(\varphi)\Big].
\end{align}

From the expression of $\Theta_{\frac{1}{\alpha}}(\varphi)$ in \eqref{redef-Tha}, we have
\[
\cos\frac{\pi-\varphi}{1+\alpha}-\Theta_{\frac{1}{\alpha}}(\varphi)=\frac{(1+\alpha)\cos\frac{\pi-\varphi}{1+\alpha}\sin\frac{\pi-\varphi}{1+\frac{1}{\alpha}}-\alpha\sin\varphi}{(1+\alpha)\sin\frac{\pi-\varphi}{1+\frac{1}{\alpha}}}.
\]
Since $\pi-\varphi=\frac{\pi-\varphi}{1+\alpha}+\frac{\pi-\varphi}{1+\frac{1}{\alpha}}$, it follows that
\[
\sin\varphi = \sin(\pi-\varphi)=\cos\frac{\pi-\varphi}{1+\alpha}\sin\frac{\pi-\varphi}{1+\frac{1}{\alpha}}+\sin\frac{\pi-\varphi}{1+\alpha}\cos\frac{\pi-\varphi}{1+\frac{1}{\alpha}},
\]
which implies
\begin{align}\label{analy-minus}
	\cos\frac{\pi-\varphi}{1+\alpha}-\Theta_{\frac{1}{\alpha}}(\varphi)=\frac{\cos\frac{\pi-\varphi}{1+\alpha}\sin\frac{\pi-\varphi}{1+\frac{1}{\alpha}}-\alpha\sin\frac{\pi-\varphi}{1+\alpha}\cos\frac{\pi-\varphi}{1+\frac{1}{\alpha}}}{(1+\alpha)\sin\frac{\pi-\varphi}{1+\frac{1}{\alpha}}}.
\end{align}
Define
\begin{align}\label{def-del}
	\lambda(\varphi):=\cos\frac{\pi-\varphi}{1+\alpha}\sin\frac{\pi-\varphi}{1+\frac{1}{\alpha}}-\alpha\sin\frac{\pi-\varphi}{1+\alpha}\cos\frac{\pi-\varphi}{1+\frac{1}{\alpha}}.
\end{align}
Then, by \eqref{analy-Del}, \eqref{analy-minus} and \eqref{def-del}, we obtain
\begin{align}\label{analy-Delagain}
	\Delta(\varphi)\geq\frac{d^2(\varphi)\Theta_{\frac{1}{\alpha}}(\varphi)+\Theta_{\frac{1}{\alpha}}(\varphi)+2}{(1+\alpha)\sin\frac{\pi-\varphi}{1+\frac{1}{\alpha}}}\,\lambda(\varphi).
\end{align}

\begin{lemma}\label{lem-del}
	Fix $\alpha\geq1$. Then $\lambda(\varphi)\geq0$ for all $\varphi\in(0,\pi)$.
\end{lemma}

\begin{proof}
	A straightforward calculation for \eqref{def-del} gives the derivative of $\lambda(\varphi)$:
	\begin{align}\label{cal-del'}
		\lambda'(\varphi)=(1-\alpha)\sin\frac{\pi-\varphi}{1+\alpha}\sin\frac{\pi-\varphi}{1+\frac{1}{\alpha}}.
	\end{align}
	Since we now assume that $\alpha\geq1$, it follows that $\lambda'(\varphi)\leq0$ on $(0,\pi)$. Moreover, note that $\lambda(\pi^-)=0$. Hence $\lambda(\varphi)\geq0$ for all $\varphi\in(0,\pi)$.
\end{proof}

\begin{remark}\label{rem-4}
	Lemma~\ref{lem-del} does not hold when $0<\alpha<1$. Indeed, if $0<\alpha<1$, then by \eqref{cal-del'}, we have $\lambda'(\varphi)>0$ on $(0,\pi)$. Hence it follows from $\lambda(\pi^-)=0$ that $\lambda(\varphi)<0$ for all $\varphi\in(0,\pi)$.
\end{remark}

Now by \eqref{analy-Delagain}, it follows from Lemma~\ref{lem-del} that $\Delta(\varphi)\geq0$ on $(0,\pi)$. Hence, based on the expression of $h'(\varphi)$ in \eqref{exp-h'}, we have
\begin{align}\label{analy-h'}
	h'(\varphi)\leq0,\quad\varphi\in(0,\varphi_0)\cup(\varphi_0,\pi).
\end{align}

\subsubsection{Conclusion: range of $h(\varphi)$}
By \eqref{analy-h'}, we know that when $\alpha\geq1$, the function $h(\varphi)$ is decreasing (not necessarily strictly) on $(0,\varphi_0)$ and $(\varphi_0,\pi)$. Recall the facts of $h(\varphi)$ stated in \eqref{fac-h-0varphipi}:
\[
h(0^+)=0,\quad h(\varphi_0^-)=-\infty,\quad h(\varphi_0^+)=+\infty,\quad h(\pi^-)=1.
\]
This implies that 
\[
h(\varphi)\notin(0,1),\quad\varphi\in(0,\varphi_0)\cup(\varphi_0,\pi).
\]

Therefore, by the definition of $h(\varphi)$ in \eqref{def-h}, we conclude that $\frac{w(\varphi)}{u(\varphi)s(\varphi)}\notin(0,1)$ in the case $u(\varphi)\neq0$. This completes the second statement of Claim~\ref{clm-anaT'}.

\subsection{Proof of Remark~\ref{rem-3}}\label{proofofremark3}
The above Sections~\ref{Sec-Auxilifunc-d}--\ref{Sec-uvarphineq0} complete the proof of Claim~\ref{clm-anaT'} under the condition $\alpha\geq1$, since Lemma~\ref{lem-del} in Section~\ref{Analysisofhvarphi} requires $\alpha\geq1$.

In contrast, we are going to show that Claim~\ref{clm-anaT'} fails for $0<\alpha<1$. More precisely, we will prove the statement in Remark~\ref{rem-3}: when $0<\alpha<1$, there exists a sufficiently small $\varepsilon_0>0$ such that for every $\varphi\in(0,\varepsilon_0)$,
\[
u(\varphi)>0 \quad\text{and}\quad 0<\frac{w(\varphi)}{u(\varphi)s(\varphi)}<1.
\]

Recall the expression of $u(\varphi)$ in \eqref{exp-finalu}:
\[
u(\varphi)=\frac{2\sin\frac{\pi-\varphi}{1+\alpha}}{1+\alpha}\big[\Theta_{\frac{1}{\alpha}}(\varphi)\big]^{2\alpha+1}\big[d^2(\varphi)-1\big].
\]
Note that $\sin\frac{\pi-\varphi}{1+\alpha}$ and $\Theta_{\frac{1}{\alpha}}(\varphi)$ are strictly positive for all $\varphi\in(0,\pi)$, and by Lemma~\ref{lem-d}, $d(\varphi)>1$ for $\varphi\in(0,\varphi_0)$. It follows that
\begin{align}\label{prfrem3one}
	u(\varphi)>0\quad\text{for every }\varphi\in(0,\varphi_0).
\end{align}

Recall the definition and expression of $h(\varphi)$ in \eqref{def-h} and \eqref{exp-h}:
\[
h(\varphi)=\frac{w(\varphi)}{u(\varphi)s(\varphi)}=\big[\Theta_{\frac{1}{\alpha}}(\varphi)\big]^{\alpha}\Big[\cos\frac{\pi-\varphi}{1+\frac{1}{\alpha}}-\frac{2d(\varphi)}{d^2(\varphi)-1}\sin\frac{\pi-\varphi}{1+\frac{1}{\alpha}}\Big].
\]
Note that $\Theta_{\frac{1}{\alpha}}(\varphi)$ is strictly positive for all $\varphi\in(0,\pi)$. In addition, since $d(0^+)=+\infty$ by Lemma~\ref{lem-d}, we have
\[
\lim_{\varphi\to0^+}\frac{d(\varphi)}{d^2(\varphi)-1}=\lim_{\varphi\to0^+}\frac{\frac{1}{d(\varphi)}}{1-\frac{1}{d^2(\varphi)}}=0.
\]
Now that when $0<\alpha<1$, we have $\frac{\pi}{1+\frac{1}{\alpha}}\in(0,\frac{\pi}{2})$ and hence
\[
\lim_{\varphi\to0^+}\Big[\cos\frac{\pi-\varphi}{1+\frac{1}{\alpha}}-\frac{2d(\varphi)}{d^2(\varphi)-1}\sin\frac{\pi-\varphi}{1+\frac{1}{\alpha}}\Big]=\cos\frac{\pi}{1+\frac{1}{\alpha}}>0.
\]
Thus, by continuity, $h(\varphi)$ is strictly positive for $\varphi$ sufficiently close to $0$.

On the other hand, since $h(0^+)=0$ by \eqref{fac-h-0varphipi}, there exists a sufficiently small $\varepsilon_0>0$ such that $0<h(\varphi)<1$ for every $\varphi\in(0,\varepsilon_0)$. That is,
\begin{align}\label{prfrem3two}
	0<\frac{w(\varphi)}{u(\varphi)s(\varphi)}<1\quad\text{for every }\varphi\in(0,\varepsilon_0).
\end{align}
Therefore, combining \eqref{prfrem3one} with \eqref{prfrem3two}, we have proved the statement in Remark~\ref{rem-3} for $0<\alpha<1$, which indicates that Claim~\ref{clm-anaT'} is false for $0<\alpha<1$.

\appendix
\section*{Appendix}
\renewcommand{\thesubsection}{\Alph{subsection}}
\setcounter{subsection}{0}
\renewcommand\theequation{\thesubsection.\arabic{equation}}

\subsection{Proof of identity \eqref{cal-gam}}\label{prf-gam}
Consider the hypergeometric function
\[
{}_2F_1(-r,n+a+1;a+1;z)=\sum_{s=0}^{r}(-1)^s\frac{r!}{s!(r-s)!}\frac{(n+a+1)_s}{(a+1)_s}z^s,
\]
where the notation $(x)_n$ denotes the Pochhammer symbol:
\[
(x)_n=x(x+1)\cdots(x+n-1)=\frac{\Gamma(x+n)}{\Gamma(x)}.
\]
Expressing the hypergeometric function in terms of Gamma functions and setting $z=1$, we obtain
\begin{align}\label{expre-2F1}
	{}_2F_1(-r,n+a+1;a+1;1)=\sum_{s=0}^{r}\frac{(-1)^s\,\Gamma(r+1)\Gamma(n+s+a+1)\Gamma(a+1)}{\Gamma(s+1)\Gamma(r-s+1)\Gamma(n+a+1)\Gamma(s+a+1)}.
\end{align}
Additionally, the Chu--Vandermonde identity at $z=1$ (see, e.g., \cite[Formula~(7.16)]{Ask}) yields
\begin{align}\label{expre-2F1-Vandermonde}
	{}_2F_1(-r,n+a+1;a+1;1)=\frac{(-n)_r}{(a+1)_r}=\frac{(-1)^r\,\Gamma(n+1)\Gamma(a+1)}{\Gamma(n-r+1)\Gamma(r+a+1)}.
\end{align}
Comparing \eqref{expre-2F1} with \eqref{expre-2F1-Vandermonde} implies the identity \eqref{cal-gam}:
\[
\sum_{s=0}^{r}\frac{(-1)^s\,\Gamma(n+s+a+1)}{\Gamma(n+1)\Gamma(s+1)\Gamma(r-s+1)\Gamma(s+a+1)}=\frac{(-1)^r\,\Gamma(n+a+1)}{\Gamma(n-r+1)\Gamma(r+a+1)\Gamma(r+1)}.
\]

\subsection{Proof of identity \eqref{cal-Aalp}}\label{prf-Aalp}
Recall the definitions of $\Theta_{\alpha}(\theta)$ and $\Theta_{\frac{1}{\alpha}}(\theta)$ in \eqref{def-Tha}:
\[
\Theta_{\alpha}(\theta)=\frac{\sin\theta}{(1+\alpha)\sin\frac{\pi-\theta}{1+\alpha}},\quad\Theta_{\frac{1}{\alpha}}(\theta)=\frac{\sin\theta}{(1+\frac{1}{\alpha})\sin\frac{\pi-\theta}{1+\frac{1}{\alpha}}}.
\]
Note that $\pi-\theta=\frac{\pi-\theta}{1+\alpha}+\frac{\pi-\theta}{1+\frac{1}{\alpha}}$ and hence
\[
\sin\theta=\sin(\pi-\theta)=\sin\frac{\pi-\theta}{1+\alpha}\cos\frac{\pi-\theta}{1+\frac{1}{\alpha}}+\cos\frac{\pi-\theta}{1+\alpha}\sin\frac{\pi-\theta}{1+\frac{1}{\alpha}}.
\]
It follows that
\[
\Theta_{\alpha}(\theta)=\frac{\sin\frac{\pi-\theta}{1+\alpha}\cos\frac{\pi-\theta}{1+\frac{1}{\alpha}}+\cos\frac{\pi-\theta}{1+\alpha}\sin\frac{\pi-\theta}{1+\frac{1}{\alpha}}}{(1+\alpha)\sin\frac{\pi-\theta}{1+\alpha}}
\]
and
\[
\Theta_{\frac{1}{\alpha}}(\theta)=\frac{\alpha\sin\frac{\pi-\theta}{1+\alpha}\cos\frac{\pi-\theta}{1+\frac{1}{\alpha}}+\alpha\cos\frac{\pi-\theta}{1+\alpha}\sin\frac{\pi-\theta}{1+\frac{1}{\alpha}}}{(1+\alpha)\sin\frac{\pi-\theta}{1+\frac{1}{\alpha}}}.
\]

On the one hand,
\begin{align}\label{cal-upp-YZ}
	\begin{split}
		e^{i\frac{\pi-\theta}{1+\alpha}}-\Theta_{\frac{1}{\alpha}}(\theta)&=\cos\frac{\pi-\theta}{1+\alpha}+i\sin\frac{\pi-\theta}{1+\alpha}-\frac{\alpha\sin\frac{\pi-\theta}{1+\alpha}\cos\frac{\pi-\theta}{1+\frac{1}{\alpha}}+\alpha\cos\frac{\pi-\theta}{1+\alpha}\sin\frac{\pi-\theta}{1+\frac{1}{\alpha}}}{(1+\alpha)\sin\frac{\pi-\theta}{1+\frac{1}{\alpha}}}\\
		&=\frac{\cos\frac{\pi-\theta}{1+\alpha}\sin\frac{\pi-\theta}{1+\frac{1}{\alpha}}-\alpha\sin\frac{\pi-\theta}{1+\alpha}\cos\frac{\pi-\theta}{1+\frac{1}{\alpha}}+i(1+\alpha)\sin\frac{\pi-\theta}{1+\alpha}\sin\frac{\pi-\theta}{1+\frac{1}{\alpha}}}{(1+\alpha)\sin\frac{\pi-\theta}{1+\frac{1}{\alpha}}}.
	\end{split}
\end{align}
On the other hand,
\begin{align}\label{cal-low-YZ}
	\begin{split}
		e^{-i\frac{\pi-\theta}{1+\frac{1}{\alpha}}}-\Theta_{\alpha}(\theta)&=\cos\frac{\pi-\theta}{1+\frac{1}{\alpha}}-i\sin\frac{\pi-\theta}{1+\frac{1}{\alpha}}-\frac{\sin\frac{\pi-\theta}{1+\alpha}\cos\frac{\pi-\theta}{1+\frac{1}{\alpha}}+\cos\frac{\pi-\theta}{1+\alpha}\sin\frac{\pi-\theta}{1+\frac{1}{\alpha}}}{(1+\alpha)\sin\frac{\pi-\theta}{1+\alpha}}\\
		&=\frac{\alpha\sin\frac{\pi-\theta}{1+\alpha}\cos\frac{\pi-\theta}{1+\frac{1}{\alpha}}-\cos\frac{\pi-\theta}{1+\alpha}\sin\frac{\pi-\theta}{1+\frac{1}{\alpha}}-i(1+\alpha)\sin\frac{\pi-\theta}{1+\alpha}\sin\frac{\pi-\theta}{1+\frac{1}{\alpha}}}{(1+\alpha)\sin\frac{\pi-\theta}{1+\alpha}}\\
		&=-\frac{\cos\frac{\pi-\theta}{1+\alpha}\sin\frac{\pi-\theta}{1+\frac{1}{\alpha}}-\alpha\sin\frac{\pi-\theta}{1+\alpha}\cos\frac{\pi-\theta}{1+\frac{1}{\alpha}}+i(1+\alpha)\sin\frac{\pi-\theta}{1+\alpha}\sin\frac{\pi-\theta}{1+\frac{1}{\alpha}}}{(1+\alpha)\sin\frac{\pi-\theta}{1+\alpha}}.
	\end{split}
\end{align}
Thus by \eqref{cal-upp-YZ} and \eqref{cal-low-YZ}, we obtain the desired identity \eqref{cal-Aalp}:
\[
\frac{e^{i\frac{\pi-\theta}{1+\alpha}}-\Theta_{\frac{1}{\alpha}}(\theta)}{e^{-i\frac{\pi-\theta}{1+\frac{1}{\alpha}}}-\Theta_{\alpha}(\theta)}=-\frac{\sin\frac{\pi-\theta}{1+\alpha}}{\sin\frac{\pi-\theta}{1+\frac{1}{\alpha}}}.
\]

\subsection{Proof of identity \eqref{simpli-iden}}\label{prf-simpli-iden}
For convenience, define
\begin{align}\label{def-YZ}
	Y:=\frac{\pi-\varphi}{1+\alpha}\quad\text{and}\quad Z:=\frac{\pi-\varphi}{1+\frac{1}{\alpha}}.
\end{align}
Then we have $\pi-\varphi=Y+Z$ and hence
\[
\sin\varphi=\sin(\pi-\varphi)=\sin(Y+Z)=\sin Y\cos Z+\cos Y\sin Z,
\]
\[
\cos\varphi=-\cos(\pi-\varphi)=-\cos(Y+Z)=\sin Y\sin Z-\cos Y\cos Z.
\]

We first deal with the left hand side of \eqref{simpli-iden}:
\[
(1+\alpha)\big[\Theta_{\frac{1}{\alpha}}(\varphi)\big]^{2}-(1+2\alpha)\Theta_{\frac{1}{\alpha}}(\varphi)\cos\frac{\pi-\varphi}{1+\alpha}+\alpha.
\]
By the definition of $\Theta_{\frac{1}{\alpha}}(\varphi)$ in \eqref{def-Tha-phi}, we have
\begin{align}\label{cal-100-Tha-YZ}
	\Theta_{\frac{1}{\alpha}}(\varphi)=\frac{\sin\varphi}{(1+\frac{1}{\alpha})\sin\frac{\pi-\varphi}{1+\frac{1}{\alpha}}}=\frac{\sin Y\cos Z+\cos Y\sin Z}{(1+\frac{1}{\alpha})\sin Z}=\frac{\alpha\sin Y\cos Z+\alpha\cos Y\sin Z}{(1+\alpha)\sin Z}.
\end{align}
Thus one can show that
\[
(1+\alpha)\big[\Theta_{\frac{1}{\alpha}}(\varphi)\big]^{2}=\frac{\alpha^2\sin^2Y\cos^2Z+\alpha^2\cos^2Y\sin^2Z+2\alpha^2\sin Y\cos Y\sin Z\cos Z}{(1+\alpha)\sin^2Z}
\]
and
\begin{align*}
	(1+2\alpha)\Theta_{\frac{1}{\alpha}}(\varphi)\cos\frac{\pi-\varphi}{1+\alpha}&=\frac{\alpha(1+2\alpha)\sin Y\cos Z+\alpha(1+2\alpha)\cos Y\sin Z}{(1+\alpha)\sin Z}\cos Y\\
	&=\frac{\alpha(1+2\alpha)\sin Y\cos Y\sin Z\cos Z+\alpha(1+2\alpha)\cos^2Y\sin^2Z}{(1+\alpha)\sin^2Z}.
\end{align*}
Then we get
\begin{align}\label{cal-100-left}
	\begin{split}
		&\quad(1+\alpha)\big[\Theta_{\frac{1}{\alpha}}(\varphi)\big]^{2}-(1+2\alpha)\Theta_{\frac{1}{\alpha}}(\varphi)\cos\frac{\pi-\varphi}{1+\alpha}+\alpha\\
		&=\frac{\alpha^2\sin^2Y\cos^2Z-\alpha(1+\alpha)\cos^2Y\sin^2Z-\alpha\sin Y\cos Y\sin Z\cos Z+\alpha(1+\alpha)\sin^2Z}{(1+\alpha)\sin^2Z}\\
		&=\frac{\alpha^2\sin^2Y\cos^2Z+\alpha(1+\alpha)\sin^2Y\sin^2Z-\alpha\sin Y\cos Y\sin Z\cos Z}{(1+\alpha)\sin^2Z}.
	\end{split}
\end{align}

We now turn to the right hand side of \eqref{simpli-iden}:
\[
(1+\alpha)\Theta_{\frac{1}{\alpha}}'(\varphi)\sin\frac{\pi-\varphi}{1+\alpha}.
\]
By \eqref{cal-der-Tha-phi}, we have
\begin{align*}
	\Theta_{\frac{1}{\alpha}}'(\varphi)&=\frac{(1+\frac{1}{\alpha})\cos\varphi\sin\frac{\pi-\varphi}{1+\frac{1}{\alpha}}+\sin\varphi\cos\frac{\pi-\varphi}{1+\frac{1}{\alpha}}}{(1+\frac{1}{\alpha})^2\sin^2\frac{\pi-\varphi}{1+\frac{1}{\alpha}}}\\
	&=\frac{\alpha(1+\alpha)(\sin Y\sin Z-\cos Y\cos Z)\sin Z+\alpha^2(\sin Y\cos Z+\cos Y\sin Z)\cos Z}{(1+\alpha)^2\sin^2Z}\\
	&=\frac{\alpha^2\sin Y\cos^2Z+\alpha(1+\alpha)\sin Y\sin^2Z-\alpha\cos Y\sin Z\cos Z}{(1+\alpha)^2\sin^2Z}.
\end{align*}
This implies that
\begin{align}\label{cal-100-right}
	(1+\alpha)\Theta_{\frac{1}{\alpha}}'(\varphi)\sin\frac{\pi-\varphi}{1+\alpha}=\frac{\alpha^2\sin^2Y\cos^2Z+\alpha(1+\alpha)\sin^2Y\sin^2Z-\alpha\sin Y\cos Y\sin Z\cos Z}{(1+\alpha)\sin^2Z}.
\end{align}

Finally, comparing \eqref{cal-100-left} with \eqref{cal-100-right} yields the desired identity \eqref{simpli-iden}:
\[
(1+\alpha)\big[\Theta_{\frac{1}{\alpha}}(\varphi)\big]^{2}-(1+2\alpha)\Theta_{\frac{1}{\alpha}}(\varphi)\cos\frac{\pi-\varphi}{1+\alpha}+\alpha=(1+\alpha)\Theta_{\frac{1}{\alpha}}'(\varphi)\sin\frac{\pi-\varphi}{1+\alpha}.
\]

\subsection{Proof of identity \eqref{simpli-iden2}}\label{prf-simpli-iden2}
Using the notation \eqref{def-YZ}, by \eqref{cal-100-Tha-YZ}, the left hand side of \eqref{simpli-iden2} is
\begin{align}\label{cal-200-left}
	\begin{split}
		\frac{1+\alpha}{\sin\varphi}\Big[\Theta_{\frac{1}{\alpha}}(\varphi)-\cos\frac{\pi-\varphi}{1+\alpha}\Big]&=\frac{1+\alpha}{\sin\varphi}\Big[\frac{\alpha\sin Y\cos Z+\alpha\cos Y\sin Z}{(1+\alpha)\sin Z}-\cos Y\Big]\\
		&=\frac{\alpha\sin Y\cos Z-\cos Y\sin Z}{\sin\varphi\sin Z}.
	\end{split}
\end{align}
Recall the definition of $d(\varphi)$ in \eqref{def-d}:
\begin{align}\label{cal-300-d-YZ}
	d(\varphi)=\frac{(1+\alpha)\cos\varphi\sin\frac{\pi-\varphi}{1+\frac{1}{\alpha}}+\alpha\sin\varphi\cos\frac{\pi-\varphi}{1+\frac{1}{\alpha}}}{\sin\varphi\sin\frac{\pi-\varphi}{1+\frac{1}{\alpha}}}=\frac{(1+\alpha)\cos\varphi\sin Z+\alpha\sin\varphi\cos Z}{\sin\varphi\sin Z}.
\end{align}
Hence the right hand side of \eqref{simpli-iden2} is
\begin{align*}
		d(\varphi)\cos\frac{\pi-\varphi}{1+\frac{1}{\alpha}}-\sin\frac{\pi-\varphi}{1+\frac{1}{\alpha}}&=\frac{(1+\alpha)\cos\varphi\sin Z\cos Z+\alpha\sin\varphi\cos^2Z-\sin\varphi\sin^2Z}{\sin\varphi\sin Z}\\
		&=\frac{\alpha(\sin\varphi\cos Z+\cos\varphi\sin Z)\cos Z-(\sin\varphi\sin Z-\cos\varphi\cos Z)\sin Z}{\sin\varphi\sin Z}
\end{align*}
Note that
\[
\sin\varphi=\sin(\pi-\varphi)=\sin(Y+Z)\quad\text{and}\quad\cos\varphi=-\cos(\pi-\varphi)=-\cos(Y+Z).
\]
It follows that
\begin{align}\label{cal-300-simpli}
	\sin\varphi\cos Z+\cos\varphi\sin Z=\sin(Y+Z)\cos Z-\cos(Y+Z)\sin Z=\sin Y
\end{align}
and
\[
\sin\varphi\sin Z-\cos\varphi\cos Z=\sin(Y+Z)\sin Z+\cos(Y+Z)\cos Z=\cos Y.
\]
Thus we conclude that
\begin{align}\label{cal-200-right}
	d(\varphi)\cos\frac{\pi-\varphi}{1+\frac{1}{\alpha}}-\sin\frac{\pi-\varphi}{1+\frac{1}{\alpha}}=\frac{\alpha\sin Y\cos Z-\cos Y\sin Z}{\sin\varphi\sin Z}.
\end{align}
Therefore, comparing \eqref{cal-200-left} and \eqref{cal-200-right} yields the desired identity \eqref{simpli-iden2}:
\[
\frac{1+\alpha}{\sin\varphi}\Big[\Theta_{\frac{1}{\alpha}}(\varphi)-\cos\frac{\pi-\varphi}{1+\alpha}\Big]=d(\varphi)\cos\frac{\pi-\varphi}{1+\frac{1}{\alpha}}-\sin\frac{\pi-\varphi}{1+\frac{1}{\alpha}}.
\]

\subsection{Proof of identity \eqref{simpli-iden3}}\label{prf-simpli-iden3}
Using the notation \eqref{def-YZ}, by the expression \eqref{cal-300-d-YZ} of $d(\varphi)$, the right hand side of \eqref{simpli-iden3} is
\begin{align*}
	d(\varphi)\sin\frac{\pi-\varphi}{1+\frac{1}{\alpha}}+\cos\frac{\pi-\varphi}{1+\frac{1}{\alpha}}&=\frac{(1+\alpha)\cos\varphi\sin Z+\alpha\sin\varphi\cos Z}{\sin\varphi}+\cos Z\\
	&=\frac{(1+\alpha)(\sin\varphi\cos Z+\cos\varphi\sin Z)}{\sin\varphi}.
\end{align*}
Then, by \eqref{cal-300-simpli}, we have
\[
d(\varphi)\sin\frac{\pi-\varphi}{1+\frac{1}{\alpha}}+\cos\frac{\pi-\varphi}{1+\frac{1}{\alpha}}=\frac{(1+\alpha)\sin Y}{\sin\varphi}=\frac{(1+\alpha)\sin\frac{\pi-\varphi}{1+\alpha}}{\sin\varphi},
\]
which is exactly the left hand side of \eqref{simpli-iden3}. Hence we obtain the desired identity \eqref{simpli-iden3}.

\end{document}